\documentclass[11pt]{article}        
\pdfoutput = 1
 \usepackage{amssymb}
\usepackage{graphicx}
\usepackage{amsmath,amsthm}

\newtheorem{theorem}{Theorem}
\newtheorem{proposition}{Proposition}
\newtheorem{lemma}[theorem]{Lemma}
\newtheorem{corollary}[theorem]{Corollary}

\newtheorem{remark}{Remark}

\def\BBN {{\mathbb N}}

\def\BBR {{\mathbb R}}
\def\BBC {{\mathbb C}}
\def\BBD {{\mathbb D}}
\def\BBT {{\mathbb T}}
\def\BBH{{\mathbb H}}
\def\mod{{mod}}

\title{\bf{Random Conformal Weldings at criticality}}
\author{\vspace{5mm}
Nicolae Tecu}\begin{document}

\date{May 2012}

\maketitle

\begin{abstract}
\addcontentsline{toc}{section}{Abstract}
\setcounter{page}{3}
We construct a family of random Jordan curves in the plane by welding together two disks on their boundaries using a random homeomorphism. This homeomorphism arises from a random measure whose density, in a generalized sense, is the exponentiated Gaussian Free Field at criticality. We also introduce a representation of the Gaussian Free Field in terms of vaguelets, which may be of separate interest. 

The result extends a theorem of Astala, Jones, Kupiainen and Saksman (\cite{AJKS09}) to criticality. 

\end{abstract}

\setcounter{page}{1}
\pagenumbering{arabic}

\section{Introduction}

We extend a theorem of Astala, Jones, Kupiainen and Saksman (\cite{AJKS09}) to criticality. The authors of \cite{AJKS09} constructed a family of random Jordan curves in the plane by solving the conformal welding problem with a random homeomorphism. The random homeomorphism arises from a random measure which is constructed by a limiting process via the exponentiated Gaussian Free Field (its restriction to the unit circle). The algorithm depends on one parameter, inverse temperature, and provides a random Jordan curve for each inverse temperature less than a certain critical value. 

In this work we extend the construction to criticality. While Astala, Jones, Kupiainen and Saksman used a white noise representation for the Gaussian Free Field, we use a vaguelet representation.

Over the last decades there has been an interest in conformally invariant fractals which could arise as scaling limits of discrete random processes in the plane. We provide an instance of such fractals in the form of Jordan (simple, closed, locally connected) curves. 

One of the most important example of conformally invariant fractals is Schramm-Loewner Evolution (introduced by Schramm, \cite{Sch00}). One version describes a random curve evolving in the disk from a point on the boundary to an interior point. The second version describes a curve evolving in the upper half plane from one point on the boundary to another point on the boundary (typically infinity). The construction depends on a parameter $\kappa\in[0,8]$. It has been shown that the scaling limits of some discrete processes are $SLE_\kappa$ curves. For example, Lawler, Schramm and Werner (\cite{LSW04}) proved that the loop erased random walk converges to $SLE_2$. Another example is the percolation exploration process which converges to $SLE_6$  on a triangular lattice as proven by Smirnov (\cite{Sm01}). For more information on $SLE_\kappa$ and an introduction to percolation we refer the reader to G. Lawler's book (\cite{L05}) and W. Werner's notes(\cite{W09}).

Unlike $SLE_\kappa$, the random curves constructed  in the present work (and in \cite{AJKS09}) do not evolve in the plane and are closed. Our curves come in two varieties. The first is the result of welding a deterministic disk to a second one by a random homeomorphism which arises by exponentiating the Gaussian Free Field. The second variety is the result of welding two disks on their boundaries by a random homemorphism which arises from two independent Gaussian Free Fields.

The subcritical case (inverse temperature less than the critical value) has also been studied by Sheffield (\cite{S10}). Using different methods he welded together two disks using two independent Gaussian Free Fields to obtain $SLE_\kappa$ (where $\kappa$ equals twice inverse temperature). Starting with a GFF on a disk, Sheffield defines a random area measure on the disk. The conformal equivalence class of the disk and random area measure is a model for random surfaces known as Liouville quantum gravity. The surfaces are called quantum surfaces. Sheffield welded together two quantum surfaces of normalized quantum area by matching quantum length on the boundaries. The two GFF version of Astala, Jones, Kupiaien, Saksman welds together two quantum surfaces by matching normalized quantum length on their boundaries. It is not clear whether the two GFF version of \cite{AJKS09} are $SLE_\kappa$ curves. Binder and Smirnov claim that the two sided version of the subcritical Jordan curves look like $SLE_\kappa$ when the local dimension of the curves is considered (\cite{BS12}). In addition, they proved that the one sided version is not $SLE_\kappa$ (reported by Sheffield in \cite{S10}).




To construct the closed random curves we follow the broad framework developed in \cite{AJKS09}. We define a random homeomorphism by exponentiating (and normalizing) the Gaussian Free Field. We then solve the conformal welding problem for this homeomorphism. While the approach is the same as in \cite{AJKS09}, the details are subtantially different.

The Gaussian Free Field is a random distribution of great importance in statistical physics and has been studied extensively in both the physics and the mathematics literature. For a mathematical introduction see \cite{S07} as well as the introductions of \cite{SS09} and \cite{SS10}. In \cite{AJKS09}, the authors used a white noise representation for the GFF. We use a vaguelet representation which was suggested to us by Peter W. Jones. Vaguelets are functions very similar and related to wavelets and appear, for example, in the work of Donoho (\cite{D95}) and Meyer and Coifman (\cite{MC97}). The vaguelet representation allows us to work with the Gaussian Free Field scale by scale. 

The procedure of exponentiating the Gaussian Free Field to get a random measure and then a homeomorphism appears also in the work of Sheffield (\cite{S10}) and Duplantier and Sheffield(e.g. \cite{DS11}). The measure constructed is a type of multifractal multiplicative cascade as introduced by Mandelbrot (\cite{M74}) and studied, among others, by Kahane (\cite{K85}), Kahane and Peyriere (\cite{KP76}), Bacry and Muzy(\cite{BM03}) and Robert and Vargas(\cite{RV08}). Multiplicative cascades appear also in mathematical finance as an important part of the Multifractal Model for Asset Returns introduced by Mandelbrot, Fisher and Calvet (\cite{MFC97}).

Most of the effort in this work is spent on solving the conformal welding problem. The classical result on this topic is that quasi-symmetric homeomorphisms are welding maps (see e.g. \cite{LV73}) and thus give rise to conformal weldings. Lehto (\cite{L70}) and David (\cite{D88}) relaxed this assumption. Oikawa (\cite{O61}) and Vainio (\cite{V85}, \cite{V89} and \cite{V95}) provided examples of homeomorphisms which are not welding maps. Hamilton relaxed the concept of welding map (introducing the generalized welding) in \cite{H91} and Bishop proved that every homeomorphism is almost a conformal welding in the sense that one can modify it on a small set and make it a welding map (see \cite{B06}). We refer the reader to \cite{H91} and \cite{B06} for more information on the welding problem and related topics. In our setting the random homeomorphism fails to be quasi-symmetric and we have to employ tools developed by Lehto in \cite{L70} to get a welding. 

We also prove that the random curves we construct are unique up to Moebius transformations. The uniqueness is equivalent to the uniqueness of the solutions of the Beltrami equation (see section \ref{SectionConformalWelding} for more details). While for quasi-symmetric homeomorphisms this follows from the measurable Riemann mapping theorem, in our case it is a consequence of a deep theorem by Jones and Smirnov (\cite{JS00}) and its extension by Nienminen and Koskela (\cite{KN05}) on conformal removability. There are cases of "wild" homeomorphisms where the welding is far from unique (see \cite{B94} and \cite{B06}).

The present work proceeds as follows: In section \ref{SectionResults} we state the main theorems. In section \ref{SectionConformalWelding} we introduce the conformal welding problem and the approach we use to solve it. We also state the main probabilistic estimate we need to prove the theorems and solve the conformal welding problem. Section \ref{SectionVaguelets} presents the construction of the vaguelets and proves some of their properties. Section \ref{SectionGFFviaVaguelets} introduces the Gaussian Free Field and its representation in terms of Fourier series, vaguelets and white noise. Section \ref{SectionRandomMeasure} outlines the construction of the random measure, its properties and presents also the proofs of these properties.  The last part gives a proof of the main probabilistic estimate by describing the decoupling and estimating the distributional properties of the random variables involved (section \ref{SectionDecoupling}). It then describes the construction of the random tree and its survival properties(section \ref{SectionStoppingTime}), gives the proof of the main modulus estimate we need (section \ref{SectionModulusEstimate}) and completes the probabilistic estimates. 

{\bf Acknowledgments }The present work was done during my graduate studies at Yale University. I am grateful to my thesis advisor, Peter W. Jones, for proposing this problem and for his support during my PhD. I am also grateful to Ilia Binder for several discussions on the topic.

\section{Results}\label{SectionResults}

Following \cite{AJKS09} we can define the restriction of the (random distribution) Gaussian Free Field on the circle by:
\begin{equation}
X = \sum_{n=1}^{\infty}\frac{A_n\cos(2\pi n \theta) + B_n\sin(2\pi n \theta)}{\sqrt{n}}
\end{equation}
where $A_n,  B_n \sim N(0,1)$ are independent.
In what follows we will use the alternative representation
\begin{equation}\nonumber
X = \sum_{I}A_I \psi_I
\end{equation}
where $A_I\sim N(0,1)$ are independent and $\{\psi_I\}$ are periodized half-integrals of wavelets. They are called vaguelets and appear for example in Donoho (\cite{D95}) and Meyer and Coifman (\cite{MC97}). Vaguelets satisfy essentially the same properties as wavelets.  

We can now consider a sequence $t_{n}$ and define the following random measure
\begin{equation}
d\nu := \lim_{k\rightarrow\infty}e^{\sum_{|I|\geq 2^{-k}}(a_I\psi_I(\theta) - t_{\log{\frac{1}{|I|}}}\psi_I^2(\theta)/2)}d\theta
\end{equation}
where $a_I\sim N(0, t_{\log{\frac{1}{|I|}}})$. 
Kahane proved in \cite{K85} that if $t_n = t < t_c=2$ this limit exists, is a non-zero, finite, non-atomic singular measure almost surely. 

We extend the result to certain increasing sequences $t_n\rightarrow t_c =2$. More precisely, given $\gamma < 1$ we take $t_n = t_c - k^{-\gamma}$ for all $n\in [n_k,n_{k+1})$, where $n_k \sim (k+1)^{\gamma}e^{C(k+1)^{3\gamma (k+1)^\gamma}/\epsilon}$,  $C$ is a large constant and $\epsilon>0$ small. 

The next step is to define the random homeomorphism $h:\BBT\rightarrow \BBT$:
\begin{equation}\label{weldinghomeo}
h(\theta) := \nu([0,\theta))/\nu([0,2\pi)) \mbox{ for } \theta\in [0,2\pi).
\end{equation}
Our goal is to solve the conformal welding problem for this homeomorphism: we seek two Riemann mappings $f_+:\BBD\rightarrow \Omega_+$ and $f_{-}:\BBD_{\infty}\rightarrow \Omega_-$ onto the complement of a Jordan curve $\Gamma$ such that $h = (f_+)^{-1}\circ f_-$. 

\begin{theorem}
Almost surely, formula \ref{weldinghomeo} defines a continuous circle homeomorphism, such that the welding problem has a solution $\Gamma$. The curve $\Gamma$ is a Jordan curve and bounds a domain $\Omega = f_{+}(\BBD)$ with Riemann mapping $f_{+}$ having the modulus of continuity better than $\delta(t)= e^{-(\log\frac{1}{t})^{1-2\epsilon}}$. For a given realization $\omega$, the solution is unique up to Moebius transformations. 
\end{theorem}

As mentioned above, this theorem is an extension of the result of Astala, Jones, Kupiainen, Saksman (\cite{AJKS09}). In their paper, the variances $t_n$ were all equal to a value strictly less than the critical one. In addition, the Riemann mapping $f_{+}$ was Hoelder continuous. By contrast, the variances here tend to the critical value, and the Riemann mapping satisfies a weaker modulus of continuity.  

The construction presented here can be used to construct weldings also for sequences $t_k$ which converge to $t_c=2$ faster than in the theorem. However, the modulus of continuity that we obtain in those cases will be worse than $e^{-\sqrt{\log\frac{1}{t}}}$ and we cannot prove the uniqueness of the welding up to Moebius transformations. Uniqueness, in our setting, is a consequence of the removability results of Jones and Smirnov (\cite{JS00}) and Koskela and Nieminen(\cite{KN05}) for conformal mappings. These results provide a sufficient modulus of continuity for conformal removability, which, in turn, implies uniqueness of the welding. It is not known what the optimal modulus of continuity is that ensures removability (see the papers \cite{JS00}, \cite{KN05}, as well as the paper by Jones and Makarov on harmonic measure \cite{JM95}).

Consider now two independent Gaussian Free Fields and construct two independent random homeomorphisms $h_1$ and $h_2$. The proof of the following theorem is the same as of the first.

\begin{theorem}
Almost surely, formula \ref{weldinghomeo} for two independent GFFs defines continuous circle homeomorphisms, such that the welding problem for homeomorphism $h_1\circ h_2^{-1}$ has a solution $\Gamma$. The curve $\Gamma$ is a Jordan curve and bounds a domain $\Omega = f_{+}(\BBD)$ with Riemann mapping $f_{+}$ having the modulus of continuity better than $\delta(t)= e^{-(\log\frac{1}{t})^{1-2\epsilon}}$. For a given realization $\omega$, the solution is unique up to Moebius transformations. 
\end{theorem}

The second theorem is related to a result of Sheffield (\cite{S10}). He proved that welding two disks using subcritical GFFs yields $SLE_\kappa$ with $\kappa =2 t$. However, there are two differences. The first difference is that he welds two disks after having normalized their random area (defined in a similar way to the measure $\nu$ above). In this construction, as well as in \cite{AJKS09}, one first normalizes the random length of the boundary. The second difference is that Sheffield deals with the subcritical case and we deal with the critical one. One can ask whether our construction produces curves which are $SLE_4$. The answer is probably no. If the Astala, Jones, Kupiainen, Saksman construction yields $SLE_\kappa$, our curves are perturbations of different $SLE_\kappa$ curves at different scales. However, as Peter Jones has pointed out in a private communication, these do not converge to $SLE_4$. 

Astala, Jones, Kupiainen and Saksman prove their theorem by first proving estimates on the random measure and then solving a degenerate Beltrami equation using a criterion of Lehto (see section \ref{SectionConformalWelding} for references and details). In order to apply this criterion they have to decouple the distortion and prove that the different scales behave roughly independently. 

In the present work we combine the properties of the vaguelets with a martingale square function argument to prove the existence of the random measure $\nu$. We also describe several of its properties. We then do a decoupling, a stopping time argument and a modulus estimate to ensure that with high proability the distortion in the Beltrami equation does not diverge too fast. Finally, one has to do a careful decoupling of the Gaussian Free Field via vaguelets to show that the scales are practically independent.

\section{The welding problem} \label{SectionConformalWelding}

Our goal is to contruct a random Jordan curve in the plane. We accomplish this by solving the conformal welding problem with a random homeomorphism. Let $\BBT$, $\BBD$ and $\BBD_{\infty}$ be the unit circle, open unit disk and the complement of the closed unit disk respectively. The conformal welding problem is as follows. 

Let $\phi$ be a homeomorphism on $\BBT$. We seek two Riemann mappings $f_+:\BBD\rightarrow \Omega_+$ and $f_{-}:\BBD_{\infty}\rightarrow \Omega_-$ onto the complement of a Jordan curve $\Gamma$ such that $\phi = (f_+)^{-1}\circ f_-$. 

Our strategy is to find these mappings by solving the Beltrami equation. Assume that $\phi = f|_\BBT$ where $f\in W^{1,2}_{loc}(\BBD,\BBD)\cap C(\BBD)$ is a homeomorphism and a solution of the Beltrami differential equation
\begin{equation}\nonumber
\frac{\partial f}{\partial \overline{z}} = \mu(z)\frac{\partial f}{\partial z}, z\in \Omega = \BBD.
\end{equation}
$\mu$ is called Beltrami coefficient. 
If the following Beltrami equation
\begin{equation}\nonumber
\frac{\partial F}{\partial \overline{z}} = 1_{\BBD}(z)\mu(z)\frac{\partial F}{\partial z}, z\in \BBC.
\end{equation}
has a unique (normalized) solution $F$, we can take $f_- = F|_{\BBD_{\infty}}$. The uniqueness implies $F$ can be factored as $F = f_+ \circ f$ for a conformal mapping $f_+:\BBD \rightarrow F(\BBD)$. Then we will have $\phi = (f_+)^{-1}\circ f_-$. The road we take is now plain: starting with the homeomorphism $\phi$, we find a mapping $f$ which is an extension of $\phi$ to the disk. Then we solve the Beltrami equation on the entire plane. Two questions arise: how do we extend the function $\phi$ to the disk? and can we solve the Beltrami equation for that extension? 

It is part of classical complex function theory that the Beltrami equation admits unique (normalized) solutions whenever $||\mu||_\infty< 1$ ($\mu$ is called uniformly elliptic). An extensive reference on the topic is \cite{AIM09}. 

The homeomorphic solutions (in $W^{1,2}(\BBC)$) of the uniformly elliptic Beltrami equation are called {\it quasiconformal} mappings and have many interesting properties. While a conformal mapping maps infinitesimal disks to infinitesimal disks, quasiconformal mappings map infinitesimal disks to infinitesimal ellipses. The Beltrami coefficient $\mu$ is also called "ellipse field" and describes this infinitesimal correspondence. 

An important quantity associated with a quasiconformal mapping/Beltrami coefficient is denoted by  $K = \frac{1+|\mu|}{1-|\mu|}$ and is called the distortion of the mapping. If $||\mu||_\infty< 1$, the distortion is bounded. 

An equivalent definition of quasiconformality is in terms of moduli of annuli. The modulus of an annulus can be defined in several equivalent ways. We first introduce the modulus of a family of curves. Given a family of locally rectifiable curves $\Gamma$ in domain $\Omega\subset\BBC$ the (conformal) modulus is the quantity 
\begin{eqnarray}
mod(\Gamma) = mod_ \Omega(\Gamma) = \inf\{\int_\Omega \rho^2 dA(x,y)\}
\end{eqnarray}
where the infimum is over all {\it metrics} $\rho:\Omega\rightarrow[0,\infty]$ such that $\int \rho ds\geq 1$ for all curves $\gamma\subset\Gamma$. Such metrics are called {\it admissible}.

The modulus is a conformal invariant: $mod(\Gamma) = mod(F(\Gamma))$ for any conformal $F$. For quasiconformal mappings with distortion $K$ we have $mod(F(\Gamma))/K\leq mod(\Gamma)\leq Kmod(\Gamma)$ for any family of curves $\Gamma$. In fact, the previous definition is equivalent to the almost invariance of the conformal modulus (for more details see the book \cite{LV73}). 

The modulus of an annulus is given by $\mod(A(r, R)) = \frac{1}{2\pi}\log(\frac{R}{r})$ and it is the same as the modulus of the family of curves which separates the two pieces of the complement of $A$. A third equivalent definition of quasiconformality is in terms of moduli of topological annuli: a mappings $F$ is quasiconformal if and only if the modulus of topological annuli is distorted by at most a multiplicative factor $K$.

The mappings $\phi$ which can be written as $f|_{\BBT}$ with $\mu$ uniformly elliptic are called quasisymmetric and satisfy $K = \sup_{s,t}\frac{|\phi(t+s)|-\phi(t)|}{|\phi(t-s) - \phi(t)|} < \infty$.

In our setting (and the one of Astala, Jones, Kupiaien, Saksman), the mapping $\phi$ is not quasisymmetric. While it can be written as the restriction of a mapping $f$, the corresponding $\mu$ is not uniformly elliptic. The mapping $f$ is called {\it degenerate} quasiconformal because the distortion, while finite almost everywhere, is unbounded.

O. Lehto (\cite{L70}) proved a very general criterion which ensures the existence of solutions to the Beltrami equation in the degenerate case.

Define:
\begin{equation}
L_K(z,r,R) =\int_r^R \frac{1}{\int_0^{2\pi} K(z + \rho e^{i\theta})d\theta}\frac{d\rho}{\rho}
\end{equation}

\begin{theorem}[\cite{AIM09}, p.584]
Suppose $\mu$ is measurable, compactly supported and satisfies $|\mu(z)|<1$ almost everywhere on $\BBC$. If the distortion function $K = \frac{1+|\mu|}{1-|\mu|}$ is locally integreable and for some $R_0$, the Lehto integral satisfies:
\begin{equation}
L_K(z,0,R_0) = \infty,  \forall z\in \BBC
\end{equation}
then the Beltrami equation 
\begin{equation}\nonumber
\frac{\partial f}{\partial \overline{z}} = \mu(z)\frac{\partial f}{\partial z}, z\in \BBC.
\end{equation}
 has a homeomorphic solution in $W^{1,1}_{loc}$.
\end{theorem}

We will be using a version of this theorem to prove the main result.  Lehto's result says that if around every point one can find an infinite number of annuli whose conformal modulus is not distorted much by a mapping with Beltrami coefficient $\mu$, then the Beltrami equation has a solution. We will prove that we can indeed find an infinite number of "good" annuli (i.e. annuli which are not distorted too much) around every point. 

Astala, Jones, Kupiainen, Saksman worked directly with the Lehto integral above. We will work with moduli of annuli. 

To solve the welding problem (i.e. prove the existence) it suffices to have the local uniqueness of the solution of the Beltrami equation. This we have because the distortion degenerates as we move closer to $\BBT=\partial\BBD$ and it is bounded inside the unit disk.

The uniqueness of the welding up to Moebius transformations is equivalent to the global uniqueness of the solution of the Beltrami equation. The information we have is not enough to draw this conclusion immediately. We have to apply a deep theorem of Jones, Smirnov(\cite{JS00}) /Koskela, Nieminen(\cite{KN05}) on conformally removable curves.  

A curve $\Gamma$ is called {\it conformally removable} if every global homeomorphism which is conformal off $\Gamma$ is automatically conformal on $\BBC$. If we knew the curves we contruct were conformally removable then we would automatically get the uniqueness of the welding, as any two weldings give rise to a global homeomorphism which is conformal off the curve $\Gamma$. 

Jones, Smirnov(\cite{JS00}) /Koskela, Nieminen(\cite{KN05}) gave sufficient conditions on the modulus of continuity on the Riemann mapping $f_+$  that ensure conformal removability. As long as the modulus of continuity of $f_+$ is better than $t\rightarrow e^{-C\sqrt{\log \frac{1}{t}}}$ (for a large constant $C$), then the curve is removable.

We stated in section \ref{SectionResults} that we can solve the welding problem with homeomorphism $\phi = h_1\circ h_2^{-1}$. The reader is encouraged to think about this as follows (in this way the proof will be exactly the same as for one GFF). For each $h_i$ consider extension $f_i$. For $h_1$ the extension is to the unit disk $\BBD$. For $h_2$ the extension is to the complement of the unit disk $\BBD_\infty$. Each extension $f_i$ has a Beltrami coefficient $\mu_i$. We seek a solution $F$ to the Beltrami equation with coefficient:

\begin{equation*}
\mu(z) =
\begin{cases}
 \mu_1(z) & \text{if } z\in \BBD,\\
\mu_2(z) & \text{if } z\in\BBD_\infty
\end{cases}
\end{equation*}

By the (local) uniqueness of the solutions to the Beltrami equation there is a conformal map $f_+:\BBD\rightarrow F(\BBD)$ such that $F = f_+\circ f_1$ on $\BBD$ and a conformal map $f_{-}:\BBD_\infty\rightarrow F(\BBD_\infty)$ such that $F = f_{-}\circ f_2$ on $\BBD_\infty$. Then we have $f_+\circ h_1 = f_{-}\circ h_2$ on $\partial\BBD$.

\subsection{The welding problem in our setting}
\label{SectionRandomConformalWelding}

In section \ref{SectionRandomMeasure} we construct a random measure and we prove, among other properties, that it is almost surely not zero on any interval. This allows us to define a random homeomorphism $\phi$ on $\BBT$. Our goal is to solve the conformal welding problem for this homeomorphism.

As espressed in the introduction the broad framework is the one from \cite{AJKS09}. Let $h(x) = \frac{\nu([0,x])}{\nu([0,1])}, x\in [0,1)$ where $\nu$ is the measure we construct in \ref{SectionRandomMeasure}. Extend $h$ periodically to $\BBR$ by setting $h(x+1) = h(x) + 1$. Extend $h$ to the upper half plane by setting (following Ahlfors-Beurling; see e.g. \cite{AIM09}) 
\begin{equation}
F(x+iy) = \frac{1}{2}\int_0^1 (h(x+ty)+h(x-ty))dt + i \int_0^1(h(x+ty) - h(x-ty))dt
\end{equation}
for $0<y<1$. This function equals $h$ on the real axis and it is a continuously differentiable homeomorphism. For $1\leq y\leq 2$ define $F(z) = z + (2-y)c_0$, where $c_0=\int_0^1 h(t)dt-1/2$. For $y>2$ define $F(z) = z$. We also have $F(z+k) = F(z)+k$.
On the unit circle we define the random homeomorphism: 
 $$
 \phi(e^{2\pi i x}) = e^{2\pi i h(x)}.
 $$
and the mapping:
 $$
 \Psi(z) = \exp(2\pi i F(\log z /2\pi i )), z\in \BBD
 $$
 is the extension of $\phi$ to the disk. The distortions of $F$ and $\Psi$ are related by 
$$
K(z, \Psi) =K (w, F), \ z = e^{2\pi i w}, w\in \BBR_+^2.
$$
 
In addition define
$$
\mu(z) = \frac{\partial_{\overline{z}}\Psi}{\partial_z \Psi}, z\in \BBD;\ \mu(z) = 0, z\notin\BBD.
$$
Our goal is to solve the Beltrami equation with this Beltrami coefficient.

We give now an upper bound for the distortion of $\mu$ that we will use. We also introduce necessary notation. 

Let $\mathcal{D}_n$ be the collection of dyadic intervals of size $2^{-n}$. For a dyadic interval $I$, let $j(I)$ be the union of $I$ and it's to neighbors of the same size. Set $C_I = \{(x,y)| x\in I, 2^{-n-1}\leq y \leq 2^{-n}\}$. Following \cite{AJKS09} let $\bold{J} = \{J_1,J_2\}$ and set
\begin{equation}
\delta_{\nu}(\bold{J}) = \frac{\nu(J_1)}{\nu(J_2)}+\frac{\nu(J_2)}{\nu(J_1)}
\end{equation}
In addition, define 
\begin{equation}
\mathcal{J}(I) = \{\bold{J} = (J_1,J_2): J_i\in\mathcal{D}_{n+5}, J_i\subset j(I)\}
\end{equation}
and
\begin{equation}
K_{\nu}(I) = \sum_{\bold{J}\subset \mathcal{J}(I)} \delta_{\nu}(\bold{J})
\end{equation}

The distortion of $\Psi$ is the same as the distortion of $F$ (the points are mapped appropriately). In the upper half of the square with base $I$(denoted by $C_I$) the distortion of $F$ is bounded by $C_0 K_\nu(I)$, for a universal constant $C_0$. As a consequence, studying the distortion of $\mu$ is really about studying the doubling properties of the random measure $\nu$ in $j(I)$. In the rest of this paper we will only use $K_\nu$.

\subsection{Main probabilistic estimate}\label{SectionMainProbabilisticEstimate}

We want to prove that almost surely we can find infinitely many annuli around each point on the unit circle which are not distorted much by a mapping with Beltrami coefficient $\mu$. We will need the following theorem:
\begin{theorem}
There are sequences $\rho_n,\tilde\rho_n, N_n, b_n, c_n$ such that:
\begin{equation}\label{MainProbEstimate}
P(\sum_{i=1}^{N_n}Mod(G(A(z,\tilde\rho_n\rho_n^{i},2\tilde\rho_n\rho_n^{i}))) < c_n N_n)\leq \tilde\rho_n\rho_n^{(1+b_n)N_n} 
\end{equation}
 for any $z\in \BBT$ and any mapping $G$ with Beltrami coefficient $\mu$.
\end{theorem}

We apply this theorem to a net of points on the unit circle and use the Borel-Cantelli theorem to get the desired statement that almost surely around every point on the unit circle there are infinitely many annuli which are not distorted by much. 

This result replaces the Lehto estimate (theorem 4.1) from Astala, Jones, Kupiainen, Saksman(\cite{AJKS09}). While their estimate covered scales one to $\rho^n$, this estimate deals with the scales in chunks. 

The theorem is a statement about distortion. 
Ideally the distortion in one scale would be independent of the distortion in another scale. However, this is not the case, each scale being correlated with every other scale. Fortunately, the correlations decay exponentially.

The setting here is more complicated than the one in \cite{AJKS09}. They used a representation of the Gaussian Free Field in terms of white noise $W$ which had a simpler correlation structure.

\subsection{Solution of the random welding problem}
\begin{theorem}\label{SolutionRandomWeldingProblem}
Almost surely there exists a random homeomorphic $W_{loc}^{1,1}-$ solution $f:\BBC\rightarrow\BBC$ to the Beltrami equation $\partial_{\overline{z}}f  = \mu \partial_{z}f$, which satisfies $f(z) = z+o(1)$ as $z\rightarrow \infty$ and whose restriction to $\BBT$ has the modulus of continuity $\omega(t)\leq e^{-(\log\frac{1}{t})^{1-2\epsilon}}$.
 \end{theorem}

 \begin{proof}
The proof is essentially the same as in \cite{AJKS09}. In this proof we use estimate (\ref{MainProbEstimate})

We start by considering for each $n$ an $[\tilde\rho_n^{-1}\rho_n^{-(1+b_n/2)N_n}]=:r_n-$ net of points on $[0,1]$ and denote $\zeta_{n,k} = \exp(2\pi i k/r_n)$ for $k\in \{1,\ldots, r_n\}$.  Set also $G_n = \{\zeta_{n,1},\ldots, \zeta_{n,r_n}\}$. Any other point on $\BBT$ is at distance at most $\sim \tilde\rho_n \rho_n^{(1+b_n/2)N_n}$ from $G_n$.
Define the event 
$$
A_{n,k}:=\{Mod(F(A(\zeta_{n,k},\tilde\rho_n \rho_n^{(1+b_n/2)N_n},\tilde\rho_n))) < c_nN_n\}
$$
Now set $A_n = \cap_k A_{n,k}$. Since
$$
\sum_{n=1}^{\infty} P(A_n) \leq \sum_n \sum_{k=1}^{r_n}P(A_{n,k})\leq \sum_n r_n \tilde\rho_{n}\rho_n^{(1+b_n)N_n} \leq \sum_n \rho_n^{N_n b_n/2}\leq \sum_n \frac{1}{2^{N_n/2}}<\infty.
$$
Borel-Cantelli tells us that almost every realization $\omega$ is in the complement of $\cup_{n>n_0(\omega) A_n}$. 

Consider the approximations $\mu_l= \frac{l}{l+1}\mu$ to $\mu$. For each $l$ denote by $f_l$ the normalized (random) solution of the Beltrami equation with coefficient $\mu_l$ and such that $f_l(z) = z+ o(1)$ as $z\rightarrow \infty$. In other words $f_l$ is a quasiconformal homeomorphism of $\BBC$. This solution is obtained by means of the measurable Riemann mapping theorem (see \cite{A06} )

We want to prove that almost surely the family $\{f_l\}$ is equicontinuous. Outside $\BBD$ all these mappings are conformal and equicontinuity follows from Koebe's theorem. Equicontinuity inside $\BBD$ follows from the fact that at any point inside the disk the distortion is determined by the measure $\nu$ on finitely many intervals and hence it is bounded. 

To prove equicontinuity on $\BBT$ we consider the functions $F_l(z) = f_l (e^{2\pi i z})$

For $P$-a.e. $\omega$  we have 
$$
Mod( F_l(A(\zeta_{n,k},\tilde\rho_{n}\rho_n^{(1+b_n/2)N_n},1))) > \sum_{i>i_0(\omega)}^n c_iN_i, \ \forall l.
$$

Lemma 2.3 in \cite{AJKS09} gives: 
$$
\frac{\mbox{diam}(F(B(\zeta,R)))}{\mbox{diam}(F(B(\zeta,r)))}\geq \frac{1}{16}\exp(\pi Mod(F(A(\zeta, r, R))))
$$ 
for any $F$ quasiconformal. 

Fix one realization $\omega$. Putting together the last two inequalities ($F_l$ is quasiconformal for any $l$).
\begin{eqnarray*}
\mbox{diam}( F_l(B(\zeta_{n,k}, \tilde\rho_{n}\rho_n^{(1+b_n/2)N_n}))) \leq 16\mbox{diam}( F_l(B(\zeta_{n,k},1))) e^{-\sum_{i=1}^n c_iN_i}e^{\sum_{i=1}^{i_0(\omega)}k_0(\omega) c_i N_i}
\end{eqnarray*}
which gives us the equicontinuity. 

Arzela-Ascoli now gives us a subsequence of $\{f_l\}$ which converges uniformly on compact sets to a function $f:\BBC\rightarrow \BBC$. 
We now show that this sequence can be picked such that $f$ is actually a homeomorphism. To this end consider the inverse functions $g_l = f_l^{-1}$. These functions satisfy the estimate: 
$$
|g_l(z) - g_l(w)|\leq 16 \pi^2 \frac{|z|^2+|w|^2+\int_{\BBD}\frac{1+|\mu_l(\zeta)|}{1-|\mu_l(\zeta)|}dA(\zeta)}{\log(e+\frac{1}{|z-w|})}
$$
Since $\frac{1+|\mu_l(\zeta)|}{1-|\mu_l(\zeta)|} \leq K(\zeta)$ and $K\in L^1(\BBD)$ (see Lemma \ref{IntegrabilityOfDistortion}) almost surely we immediately have that the sequence $\{g_l\}$ is equicontinuous. In addition, the integrability of distortion leads to the conclusion that $f\in W^{1,1}_{loc}$ (for a proof of this last fact see \cite{AIM09}, theorem 20.9.4).

The modulus of continuity is given by the relation between $\tilde\rho_n\rho_n^{(1+b_n)N_n}$ and $e^{-\sum_{i=1}^n c_iN_i}$.


\end{proof}

\section{Vaguelets}\label{SectionVaguelets}

\subsection{Construction}

Consider a wavelet basis $\{\Phi_{j,l}\}$ of $L^2(\BBR)$ with mother wavelet $\Phi:\BBR\rightarrow \BBR$.
Following Donoho (\cite{D95}) set 
\begin{equation*}
\phi(t) := \frac{1}{2\pi} \int_\BBR e^{it\omega} \widehat{\Phi}(\omega)|\omega|^{-1/2}d\omega
\end{equation*}
The vaguelet $\phi$ is the half integral of $\Phi$ and satisfies the following properties (we can choose $q$ by choosing a suitable decay for $\Phi$): 
\begin{eqnarray*}
|\phi(t)| \leq C_1(1+|t|)^{-(q+1)},\ t\in \BBR, \\
\int \phi(t) dt = 0 \\
|\phi(t)-\phi(s)| \leq C_2 |t-s|
\end{eqnarray*}
Following Y. Meyer (\cite{M90}) consider the periodized functions:
\begin{eqnarray*}
\Psi_j(\theta) := 2^{j/2}\sum_{-\infty}^{\infty} \Phi(2^j(\theta-k))
\end{eqnarray*}
The functions $\{1\}\cup\{\Psi_j(\theta-k2^{-j})\}_{j, 0\leq k< 2^j}$ form a periodic orthonormal wavelet basis of $L^2(\BBT)$.

We now introduce the periodic vaguelet. For a function on the torus with Fourier series $f(\theta)\sim \sum \widehat{f}(n) e^{2\pi i n \theta}$ we have $(-\Delta) f\sim \sum (2\pi n)^2  \widehat{f}(n) e^{2\pi i n \theta}$ so we may define the operator $(-\Delta)^{-1/4}$ by \begin{eqnarray*}
\widehat{(-\Delta)^{-1/4} f} (n) := \frac{1}{\sqrt{2\pi |n|}} \widehat{f}(n)
\end{eqnarray*}
Define now the periodic vaguelet $\psi (\theta)$ by 
\begin{equation*}
\widehat{\psi}(n) := \frac{1}{\sqrt{2\pi |n|}} \widehat{\Psi}(n)
\end{equation*}
where $\Psi$ is the periodic wavelet. 
Define 
\begin{equation}
\widehat{\psi}_{j,l} (n) := \frac{1}{\sqrt{2\pi |n|}} \widehat{\Psi_j(\cdot - l2^{-j})}(n) 
\end{equation}
The vaguelets $\psi_{j,l}$ are periodized versions of the $\phi_{j,l}$ and have essentially the same properties. While the wavelets are a basis for $L^2(\BBT)$, the vaguelets are a basis for the space $H_0^{1/2}(\BBT)$ of functions $f$ which have mean zero, and half of a derivative in $L^2(\BBT)$. In other words, $f\sim \sum_{n\neq 0} a_n e^{2\pi i n \theta}$ and the norm on the space is:
\begin{equation*}
||\sum_{n\neq 0} a_n e^{2\pi i n \theta}||_{H_0^{1/2}}=\left(\sum_{n\neq 0}2\pi n |a_n|^2\right)^{1/2}.
\end{equation*}

\subsection{Properties}

In this section we present some properties of the vaguelets that will prove useful later. We start with:

\begin{eqnarray*}
||\psi_{j,l}||_{L^2(S^1)}^2 &=& \sum_{n\neq 0}\frac{1}{2\pi |n|} |\widehat{\Phi}_{j,l}(n)|^2 \\
\int_0^1 |\psi_{j,l}(\theta)| d\theta &\leq& C 2^{-j}
\end{eqnarray*}

We give a few estimates which we will need later:
\begin{lemma}\label{DecayOfVaguelet}
For any $q>0$ and any $\theta\in \tilde{J}$ with $|\tilde{J}|=2^{-j}$ and $J=[2^{-j}l,2^{-j}(l+1)]\neq \tilde{J}$
\begin{eqnarray}
\left|\psi_J(\theta)\right|\leq \frac{C_q}{(1+|2^j\theta-l|)^{q}}\\
\left|\psi_J'(\theta)\right|\leq \frac{C_q2^j}{(1+|2^j\theta-l|)^{q-1}}
\end{eqnarray}
The expression $|2^j\theta-l|$ should be understood modulo $2^j$ and can be replaced by $\frac{dist(J,\tilde{J})}{2^{-j}}$.
\end{lemma}
\begin{remark} 
The quantity $\frac{dist(J,\tilde{J})}{2^{-j}}$ is the number of intervals of size $2^{-j}$ that separate $J$ and $\tilde{J}$ on the torus.
\end{remark}
\begin{proof}
The proof is a simple computation using the decay of the vaguelet $\phi$. 
\begin{equation}
|\psi_J(\theta)|\leq \sum_k |\phi(2^j \theta - l + 2^j k)|\leq \sum_k \frac{C}{(1+|2^j \theta -l +2^j k|)^{q+1}}.
\end{equation}
We may describe $\tilde{J}$ as the interval $[2^{-j} \tilde{l}, 2^{-j} (\tilde{l}+1)]$. Then $2^j \theta -l \in[ \tilde{l} - l, \tilde{l}+1 - l]$. $|l-\tilde{l}|$ equals the number of dyadic intervals (modulo $2^j$) of length $2^{-j}$ separating $J$ and $\tilde{J}$ on the unit torus (quantity denoted by  $\frac{dist(J,\tilde{J})}{2^{-j}}$).  This gives also the largest term in the series above. All the other terms decrease very fast and their sum is dominated by the first.  

An identical argument works for the derivative $\psi_J(\theta)$.
\end{proof}

We also have the following:
\begin{lemma}\label{VagueletLemma}
For the family of vaguelets defined above the following relations hold: there is a constant $C_0$ such that for all dyadic $I$, all $m$
\begin{eqnarray}
\left|\sum_{|J|\geq 2^{-m}}\psi_J^2(\theta) - \frac{(m+1)\ln 2}{\pi}\right|\leq C_0 \ \forall  \theta \label{VagueletInequality}\\ 
\left|\sum_{|J|\leq 2^{-m}, J\not\subset 3I}\psi_J^2(\theta) \right|\leq C_0 \ \forall \theta\in I
\end{eqnarray}
where $3I$ is the interval formed by the dyadic interval $I$ and its left and right neighbors of the same size.
\end{lemma}

\begin{proof}
We recall that $\psi_J$ are periodic vaguelets (defined on $[0,1]$). We will first prove that 
\begin{eqnarray}
\left|\int_0^1\sum_{|J|\geq 2^{-m}}\psi_J^2(\theta) d\theta- \frac{(m+1)\ln 2}{\pi}\right|\leq C_0 \\ 
\end{eqnarray}
and then deal with the pointwise estimate. We will prove this inequality by reducing the computation to the wavelets $\Phi_J$ on the line.

In the following we will replace the notation $\psi_J$ by $\psi_{j,l}$ where $J=[2^{-j}l,2^{-j}(l+1)]$. The periodic vaguelet $\psi_{j,l}$ was defined by:
\begin{equation}
\widehat{\psi}_{j,l} (n) := \frac{1}{\sqrt{2\pi |n|}} \widehat{\Psi_j(\cdot - l2^{-j})}(n) 
\end{equation}
where $\Psi$ was the periodic wavelet and $\Psi_j$ its refinement to level $2^{-j}$. We also have
\begin{equation}
 \widehat{\Psi_j(\cdot - l2^{-j})}(n)  = 2^{-j/2} e^{-2\pi i n l 2^{-j}} \widehat{\Phi}(2^{-j}n)
\end{equation}
where $\Phi$ is the mother wavelet on $\BBR$.

We then have
\begin{eqnarray}
\sum_{j=0}^m\sum_{l=0}^{2^j-1}\int_0^1\phi_{j,l}^2(\theta)d\theta = \sum_{j=0}^m\sum_{l=0}^{2^j-1}\sum_{n\neq 0}\frac{2^{-j}}{2\pi |n|}|\hat\Phi(n2^{-j})|^2= \sum_{j=0}^m\sum_{n\neq 0}\frac{1}{2\pi |n|}|\hat\Phi(n2^{-j})|^2
\end{eqnarray}

We remark that for $m$ large enough:
\begin{equation}
\sum_{n\neq 0}\frac{1}{2\pi |n|}|\hat\Phi(n2^{-m})|^2 = \sum_{n\neq 0}\frac{1}{2\pi |n|2^{-m}}|\Phi(n2^{-m})|^22^{-m}\approx \frac{1}{2\pi}\int \frac{|\hat\Phi(\xi)|^2}{|\xi|}d\xi
\end{equation}
By $\approx$ we mean equal up to a small error (and all such errors add up to at most a constant).
We will prove that 
\begin{equation}\label{ln2identity}
\frac{1}{2\pi}\int \frac{|\hat\Phi(\xi)|^2}{|\xi|}d\xi = \frac{\ln2}{\pi}
\end{equation}
This fact follows from the construction of the mother wavelet $\Phi$. We recall the construction procedure (see \cite{M90}, chapter 3, or \cite{BNB00}, chapter 7).

One considers a function $m_0(\xi)$ (also called filter) with the following properties:
\begin{itemize}
\item $m_0$ is continuous and $2\pi$ periodic.
\item $|m_0(\xi)|^2 + |m_0(\xi+\pi)|^2=1$.
\item $m_0(0) = 1$ and $m_0(\xi)\neq 0$ on $[-\frac{\pi}{2},\frac{\pi}{2}]$.
\end{itemize}
For any such filter one considers the father wavelet (aka scaling function) $f$ given by 
\begin{equation}
\hat f(\xi) := \prod_{j\in\BBN}m_0(\frac{\xi}{2^j})\label{DefinitionFatherWavelet}
\end{equation} 
and the mother wavelet will be given by the relation
\begin{equation}
\hat \Phi(2\xi) :=e^{-i\xi} \overline{m_0}(\xi +\pi) \hat f(\xi) 
\end{equation}
We may now proceed with our argument:
\begin{eqnarray}
\int \frac{|\hat\Phi(\xi)|^2}{|\xi|}d\xi = \int\frac{|m_0(\xi/2 +\pi)|^2|\hat  f(\xi/2)|^2}{|\xi|}d\xi = \int\frac{|m_0(\xi +\pi)|^2|\hat f(\xi)|^2}{|\xi|}d\xi
\end{eqnarray}

We make the following observation (using the second property of the filter $m_0$):
\begin{eqnarray}
\int_{2^j}^{2^{j+1}}\frac{|\hat f(\xi)|^2}{\xi}d\xi = \int_{2^j}^{2^{j+1}}\frac{|m_0(\xi+\pi)|^2|\hat f(\xi)|^2}{\xi}d\xi + \int_{2^j}^{2^{j+1}}\frac{|m_0(\xi)|^2|\hat f(\xi)|^2}{\xi}d\xi\\
= \int_{2^j}^{2^{j+1}}\frac{|m_0(\xi+\pi)|^2|\hat f(\xi)|^2}{\xi}d\xi + \int_{2^j}^{2^{j+1}}\frac{|\hat f(2\xi)|^2}{\xi}d\xi
\end{eqnarray}
We have used the fact that $\hat f(2\xi) = m_0(\xi)\hat f(\xi)$ (which follows from relation \ref{DefinitionFatherWavelet}). We make a change of variable in the second term to get:
\begin{eqnarray}
\int_{2^j}^{2^{j+1}}\frac{|\hat f(\xi)|^2}{\xi}d\xi = \int_{2^j}^{2^{j+1}}\frac{|m_0(\xi+\pi)|^2|\hat f(\xi)|^2}{\xi}d\xi + \int_{2^{j+1}}^{2^{j+2}}\frac{|\hat f(\xi)|^2}{\xi}d\xi
\end{eqnarray}

Adding these terms for $j\geq -n$ we get:
\begin{eqnarray}
\int_{2^{-n}}^{2^{-n+1}}\frac{|\hat f(\xi)|^2}{\xi}d\xi = \int_{2^{-n}}^{\infty}\frac{|m_0(\xi+\pi)|^2|\hat f(\xi)|^2}{\xi}d\xi 
\end{eqnarray} 
The product $ \prod_{j\in\BBN}m_0(\frac{\xi}{2^j})$ converges uniformly on bounded sets of $\BBR$ (see e.g. \cite{BNB00}) and thus $\hat f$ is continuous. Since $\hat f(0) = 1$ we get:
\begin{equation}
\int_{2^{-n}}^{2^{-n+1}}\frac{|\hat f(\xi)|^2}{\xi}d\xi \rightarrow \ln2 \mbox{ as } n \rightarrow \infty
\end{equation} 
which gives us (\ref{ln2identity}).

To prove the pointwise estimate we use the (inverse) Fourier transform:
\begin{eqnarray}
\psi_{j,l}^2(\theta) = \sum_n e^{2\pi i n \theta} \sum_{k\neq n,k \neq 0}\frac{2^{-j}e^{-2\pi i n l2^{-j}}}{\sqrt{2\pi|k|}\sqrt{2\pi |n-k|}}\hat\Phi(2^{-j}k)\hat\Phi(2^{-j}(n-k))
\end{eqnarray}
Since $\sum_{l=0}^{2^j-1}e^{-2\pi i n l2^{-j}} = 0$ for all $n\neq 0$ and $2^j$ for $n=0$ we get:
\begin{eqnarray}
\sum_{l=0}^{2^j-1}\psi_{j,l}^2(\theta) = \sum_{k \neq 0}\frac{1}{\sqrt{2\pi|k|}\sqrt{2\pi |k|}}\hat\Phi(2^{-j}k)\hat\Phi(2^{-j}(-k))\\
= \sum_{k \neq 0}\frac{2^{-j}}{2\pi|2^{-j}k|}\hat\Phi(2^{-j}k)\hat\Phi(2^{-j}(-k))\\
\approx  \frac{1}{2\pi}\int \frac{\hat\Phi(\xi)\hat\Phi(-\xi)}{|\xi|}d\xi = \frac{1}{2\pi}\hat{\Psi}*\hat\Psi(0) = \frac{1}{2\pi}\widehat{\Psi^2}(0) = \frac{\ln 2}{\pi}
\end{eqnarray}
by the computation in the previous part of the proof. 

The second inequality in the lemma is a consequence of the decay of the vaguelets (Lemma \ref{DecayOfVaguelet}). 
\end{proof}

\section{The Gaussian Free Field and vaguelets}
\label{SectionGFFviaVaguelets}

 Heuristically, the Gaussian Free Field is a Gaussian "random variable" on an infinite dimensional space. A precise and correct definition is more subtle. We first give a few facts about (usual) Gaussian random variables and then extend the concept to infinite dimensional spaces. We follow the presentation in \cite{S07}, where Sheffield gives a good introduction to the GFF.

Let $(\cdot,\cdot)$ be an inner product on $\BBR^d$ and let $\mu$ be the probability measure $e^{-(v,v)/2}Z^{-1}dm(v)$, where $m$ is Lebesgue measure on $\BBR^d$ and $Z$ is the normalizing constant. 

\begin{proposition}[\cite{S07}]
Let $v$ be a Lebesgue measurable random variable on $\BBR^d$ with inner product $(\cdot,\cdot)$. The following are equivalent:
\begin{itemize}
\item[a] $v$ has the (Gaussian) law $\mu$.
\item[b] $v$ has the same law as $\sum_{j=1}^d \alpha_j v_j$ where $v_1,\ldots, v_d$ are a deterministic orthonormal basis of $\BBR^d$ and $\alpha_j$ are i.i. d. Gaussian random variables with mean zero and variance one. 
\item[c] The characteristic function (Fourier transform) of $v$ is given by 
\begin{equation}\nonumber
E[e^{i(v,t)}] = e^{-\frac{||t||^2}{2}}
\end{equation}
\item[d] For each fixed $w\in\BBR^d$, the inner product $(v,w)$ is a zero mean Gaussian random variable with variance $(w,w)$.
\end{itemize}
\end{proposition}

The Gaussian Free Field is supposed to be a variable on the infinite dimensional space $H^1(D) = W_0^{1,2}(D)$, where $D$ is a subdomain of $\BBR^d$. If $D$ has no boundary the space $H^1(D) = W_0^{1,2}(D)$ stands for the Sobolev space of functions with mean zero and one derivative in $L^2$. This space is a Hilbert space with inner product $(f, g)_{\nabla} = \int \nabla f\nabla g$.

Ideally one would consider an orthonormal basis $b_j$ of this space and declare the GFF to be the random variable given by $\sum_{j = 1}^{\infty} \alpha_j b_j$ with $\alpha_j\sim N(0,1)$ i.i.d. However, this sum doesn't converge in $H^1$ and one has to consider its convergence in a bigger (Banach) space. 

Alternatively, one can define the GFF as being the formal sum $h =\sum_{j = 1}^{\infty} \alpha_j b_j$ with $\alpha_j\sim N(0,1)$ i.i.d. and $\{b_j\}$ an {\it ordered} orthonormal basis. For any fixed $f\in H^1$ with $f = \sum \beta_j b_j$ one can define the inner product as a random variable $(h,f)_{\nabla} = \lim_{k\rightarrow\infty}\sum_{j=1}^k \alpha_j\beta_j$.

The Gaussian Free Field then becomes a collection of mean zero Gaussian random variables $\{ (h,f)_{\nabla}\}_{f\in H^1}$ with variance $(f,f)_{\nabla}$ and covariance strucure given by 
\begin{equation}\nonumber
E[(h,f)_{\nabla}(h,g)_{\nabla}] = (f,g)_{\nabla} 
\end{equation} 

On a manifold with no boundary $D$ in $\BBR^d$ we also have $(\rho_1,\rho_2)_{\nabla} = - (\rho_1,-\Delta \rho_2)_{L^2}$. This allows us to define the GFF as being a collection of mean zero Gaussian random variables $\{ (h,\rho)\}_{\rho\in (-\Delta)H^1(D)}$ with covariance structure given by 
\begin{equation}\nonumber
E[(h,\rho_1) (h,\rho_2)] =\int_{D\times D} \rho_1(x)G(x,y)\rho_2(y)dx dy 
\end{equation} 
where $G(x,y)$ is Green's function on $D$ (the inverse of the laplacian operator on $D$).


The Gaussian Free Field we work with is the trace on $\BBT$ of the 2-dimensional GFF. In stead of being a random variable on the space $H^1$, the trace of the 2-dimensional GFF is a random variable on the space $H_0^{1/2}$ of mean zero and half a derivative in $L^2$. Formally, this can be defined as 
\begin{equation}
H_0^{1/2}(\BBT) = \{f|\int_\BBT f d\theta= 0, \left(\frac{d}{d\theta}\right)^{1/2}f\in L^2(\BBT)\}.
\end{equation}

We recall that $\{1, \sqrt{2}\cos(2\pi n\theta), \sqrt{2}\sin(2\pi n\theta)\}$ is an orthonormal basis of $L^2(\BBT)$. This implies that $\{\frac{1}{\sqrt{\pi |n|}}\cos(2\pi  n\theta), \frac{1}{\sqrt{\pi |n|}} \sin(2\pi n\theta)\}$ form an orthonormal basis of $H_0^{1/2}(\BBT)$.
This allows Astala, Jones, Kupiainen and Saksman (\cite{AJKS09}) to define the trace on $\BBT$ of the 2-dim GFF as the the random distribution:

\begin{equation}\label{FourierGFF}
X = \sum_{n=1}^{\infty}\frac{A_n\cos(2\pi nt) + B_n\sin(2\pi n t)}{\sqrt{n}}
\end{equation}
where $A_n, B_n\sim N(0,1)$ are independent.


We can also consider a wavelet basis $\{\phi_I\}$ for $L^{2}(\BBT)$. The image under  $(-\Delta)^{-1/4}$ is a vaguelet basis $\{\psi_I\}$ for $H_0^{1/2}(\BBT)$. Then we have that up to a probability preserving transformation the GFF can be rewritten as 
\begin{equation}\nonumber
X = \sum_{I}A_I \sqrt{\pi}\psi_I
\end{equation}
where $A_I\sim N(0,1)$.  The factor $\sqrt{\pi}$ appears in this expression because it is missing in definition (\ref{FourierGFF}). From this point onwards we will include it in the notation $\psi_I$.

One can see the equivalence of the representations in the following way. If we have two bases $\{F_j\},\{G_k\}$for $H_0^{1/2}$ (e.g. coming from bases $\{f_j\},\{g_k\}$ of $L^2$) then
\begin{eqnarray*}
\sum_j A_j F_j = \sum_j A_j \sum_k (F_j,G_k)_{-1/4}G_k = \sum_k G_k \sum_j (F_j,G_k)_{-1/4} A_j
\end{eqnarray*} 
Since $\sum_j  (F_j,G_k)^2_{-1/4} = 1$ for all $k$ then $\sum_j (F_j,G_k)_{-1/4} A_j \sim N(0,1)$ so
\begin{equation}\nonumber
\sum_j A_j F_j = \sum_k B_k G_k,\ B_k \in N(0,1)
\end{equation}
up to a measure preserving transformation.

In \cite{AJKS09} Astala, Jones, Kupiainen and Saksman used a white noise representation for the GFF. Gaussian white noise is a centered Gaussian process, indexed by sets of finite hyperbolic area measure in the upper half-plane and with covariance structure given by the hypebolic area measure of the intersection of sets. The trace of the Gaussian Free Field on $\BBT$ was then expressed as
\begin{eqnarray}
H(x) = W(x+H), x\in \BBT \mbox{ where } \\
H=\{(x,y)\in \BBH | -1/2<x<1/2, y>\frac{2}{\pi}\tan(|\pi x|)\}
\end{eqnarray} 
The geometry of the set $H$ allowed Astala, Jones, Kupiainen, Saksman to decouple the variables on different scales in their main probabilistic estimate. While the vaguelets are slightly more complicated, their tails decay fast enough to allow us a similar decoupling.


\section{Random measure}\label{SectionRandomMeasure}

\subsection{Construction of the measure}

We write the GFF as 
\begin{equation}
X = \sum_I a_I \psi_I(\theta)
\end{equation}
where $a_I \sim N(0,1)$ and $\{\psi_I\}$ are the vaguelets defined in section \ref{SectionVaguelets}, scaled by the factor $\sqrt{\pi}$ as we pointed out in section \ref{SectionGFFviaVaguelets}.

Take $t_k = t_c - k^{-\gamma}$ and  $n_k \sim (k+1)^{\gamma}e^{\frac{1}{\epsilon}C(k+1)^{3\gamma (k+1)^\gamma}}$. Define $S_0 :=0$ and 
\begin{equation}
 S_{k+1}(\theta) := S_{k}(\theta) + \sum_{I: |I|\in [2^{-n_{k+1}}, 2^{-(n_{k}+1)}]} \left(a_{I}\psi_I(\theta) - \frac{t_{k+1}}{2}\psi^2_I(\theta)\right), \forall k\geq 0.
\end{equation}

Here $a_{I}$ are independent centered Gaussian random variables of variance $t_{k+1}$ for $|I|\in [2^{-n_{k+1}}, 2^{-(n_{k}+1)}]$. The sequence $\{t_k\}$ is increasing to $t_{c}=2$. Define 
\begin{eqnarray}
d\nu_k := e^{S_k(\theta)}d\theta \label{definitionmeasure}\\
F_k := \int_{[0,1]} d\nu_k = \int_{[0,1]} e^{S_k(\theta)}d\theta. 
\end{eqnarray}

It is easy to see $\{F_k\}$ is an $L^1$ martingale and hence it has an almost sure limit $F_0$. We want to prove this martingale is in the space $L\log L$ to ensure the $F_k\rightarrow F_0$ in $L^1$. This and Kolmogorov's zero-one law imply $F_0$ is almost surely nonzero. The subcritical case, $t_k = t < 2$ was studied by Kahane (\cite{K85}) who proved that the martingale is in $L^p$ for some $p=p(t)>1$.

In the next result we will repeatedly use the equivalence of the $L^p$ norm of a martingale to the $L^p$ norm of its square function. Let $\{M_k\}$ with $M_0 = 0$ be an $L^p$ martingale for $1<p<\infty$. Define the martingale differences $\Delta_k := M_k-M_{k-1}$ and set $\mathcal{S}_k := \left(\sum_{i=1}^k \Delta_i^2\right)^{1/2}$. The latter is called {\it the martingale square function} and captures the $L^p$ behavior of the martingale (see \cite{B66}): 
\begin{equation}
\frac{1}{C_p}E[\mathcal{S}_k^p]\leq E[|M_k|^p]\leq C_p E[\mathcal{S}_k^p]
\end{equation}
The constant $C_p$ has order of magnitude $\frac{1}{p-1}$ and is independent of the martingale. We will apply the right inequality repeatedly in the case when $p\in(1,2]$ in which case the function $x\rightarrow x^p$ is subadditive and the inequality becomes 
\begin{equation}
 E[|M_k|^p]\leq C_p E[\sum_{i=1}^k \Delta_i^p]
\end{equation}
Heuristically, one can interpret this as saying that the martingale differences behave as if they were independent.

\begin{theorem} \label{LLogLMomentEstimate}
Take $t_k = t_c - k^{-\gamma}$ and  $n_k \sim (k+1)^{\gamma}e^{C(k+1)^{3\gamma (k+1)^\gamma}/\epsilon}$. Then the martingale $F_k$ satisfies $E[F_k\log(1+F_k)] <C$.
\end{theorem}

\begin{corollary}
The maximal function $F^*=\sup_k |F_k|$ is in $L^1$ and hence the martingale $\{F_k\}$ is an $H^1$-bounded martingale ($H^1$ denotes here the Hardy space).
\end{corollary}

\begin{proof}[Proof of theorem \ref{LLogLMomentEstimate}] We begin by obtaining $L^p$ estimates on the martingale differences. 

\begin{eqnarray*}
F_{k+1}-F_k = \int_0^1 e^{S_{k+1}(\theta)} -e^{S_k(\theta)} d\theta = \int_0^1 e^{S_k(\theta)}(e^{\sum_I \left(a_{I}\psi_I(\theta) - \frac{t_k}{2}\psi^2_I(\theta)\right)} -1)d\theta
\end{eqnarray*}
The exponent of the second term can be written as 
\begin{equation}
\sum_I \left(a_{I}\psi_I(\theta) - \frac{t_{k+1}}{2}\psi^2_I(\theta)\right) = \sum_{|J| = 2^{-(n_{k}+1)}}\sum_{I \subset J} \left(a_{I}\psi_I(\theta) - \frac{t_{k+1}}{2}\psi^2_I(\theta)\right) =: \sum_{|J| = 2^{-n_{k}+1}} A_J
\end{equation} 
(with the obvious definition). In fact, we may denote $A_J$, by $A_j$ where $j\in \{1,\ldots, 2^{n_{k}+1}\}$. Set $A_0 = 0$. We may now write: 
\begin{eqnarray*}
F_{k+1}-F_k = \int_0^1 e^{S_k(\theta)}(e^{\sum_{j=1}^{2^{n_{k}+1}} A_j}-1)d\theta = \\
 \sum_{l=1}^{2^{n_{k}+1}} \int_0^1 e^{S_k(\theta)}(e^{\sum_{j=0}^{l} A_j}-e^{\sum_{j=0}^{l-1} A_j})d\theta =\\ \sum_{l=1}^{2^{n_{k}+1}} \int_0^1 e^{S_k(\theta)}e^{\sum_{j=0}^{l-1} A_j}(e^{A_l}-1)d\theta =\\
\sum_l \int_0^1 X_l Y_l d\theta
\end{eqnarray*}
where $Y_l := (e^{A_l}-1)$ and are mutually independent and independent of $\{X_l\}$ (stand for the other two terms). The independence follows from the fact that the normal variables which appear in the definition of $A_l$ are associated with the dyadic intervals which are subsets of the dyadic interval of size $2^{-(n_k+1)}$ which corresponds to $l$. 

We also have $E[Y_l] = 0$. This implies  $\{\sum_{l=0}^{L} \int_0^1 X_l Y_l d\theta\}_0^{2^{n_k+1}}$ is a martingale with respect to increasing $L$. This implies (via the martingale square function) for $p\in(1,2]$:
\begin{eqnarray}
E[|F_k-F_{k+1}|^p] \leq c_p \sum_l E[| \int_0^1 X_l Y_l d\theta|^p]
\end{eqnarray}

We estimate the term $\int_0^1 X_l Y_l d\theta$ in the same way since it can itself be thought of as a martingale in the following way.  Denote by $J(i)$ the collection of dyadic intervals $I\subset J$ such that $|I| \geq |J|2^{-i}$. Set also $J(-1) = \emptyset $ and 
$$A_J = \sum_{I \subset J} \left(a_{I}\psi_I(\theta) - \frac{t_{k+1}}{2}\psi^2_I(\theta)\right) = \sum_{i=1}^{n_{k+1}-n_k} \sum_{I\in J(i)\setminus J(i-1)} \left(a_{I}\psi_I(\theta) - \frac{t_{k+1}}{2}\psi^2_I(\theta)\right).$$ 
We may now write:
\begin{eqnarray*}
\int_0^1 X_l Y_l d\theta = \sum_{i=0}^{n_{k+1}-n_k} \int_0^1 X_l (e^{\sum_{I\in J(i)}\left(a_{I}\psi_I(\theta) - \frac{t_{k+1}}{2}\psi^2_I(\theta)\right)}-e^{\sum_{I\in J(i-1)}\left(a_{I}\psi_I(\theta) - \frac{t_{k+1}}{2}\psi^2_I(\theta)\right)}) d\theta \\
=  \sum_{i=0}^{n_{k+1}-n_k} \int_0^1 X_l e^{\sum_{I\in J(i-1)}\left(a_{I}\psi_I(\theta) - \frac{t_{k+1}}{2}\psi^2_I(\theta)\right)}(e^{\sum_{I\in J(i)\setminus J(i-1)}\left(a_{I}\psi_I(\theta) - \frac{t_{k+1}}{2}\psi^2_I(\theta)\right)} -1)d\theta \\
=  \sum_{i=0}^{n_{k+1}-n_k} \int_0^1 X_l   Z_{l,i-1} T_{l,i} d\theta
\end{eqnarray*}
We use the index $l$ in stead of $J$ because we will sum by $l$ later and $J$ is the $l$'th dyadic interval of length $2^{-n_k-1}$.

The random variables $T_{l,i}$ are mutually independent (with respect to $i$) and are also independent of $\{X_l Z_{l,\cdot}\}_0^{i-1}$. They also have mean equal to zero. By the same argument as before:
\begin{eqnarray*}
E[| \int_0^1 X_l Y_l d\theta|^p] \leq c_p \sum_{i=0}^{n_{k+1}-n_k} E[| \int_0^1 X_l   Z_{l,i-1} T_{l,i} d\theta|^p]
\end{eqnarray*}

In $J(i)\setminus J(i-1)$ there are $2^{i}$ dyadic intervals of length $2^{-n_k-i}$ which we now index by $\chi$. The associated random variables are centered independent Gaussians. So we may write yet again:
\begin{eqnarray*}
 \int_0^1 X_l   Z_{l,i-1} T_{l,i} d\theta = \sum_{\chi=1}^{2^i} \int_0^1 X_l Z_{l,i-1}e^{\sum_1^{\chi-1} (a_{I}\psi_I(\theta) - \frac{t_{k+1}}{2}\psi^2_I(\theta))}(e^{a_{I_\chi}\psi_{I_\chi}(\theta) - \frac{t_{k+1}}{2}\psi^2_{I_\chi}(\theta)}-1)d\theta \\
 = \sum_{\chi=1}^{2^i} \int_0^1 X_l Z_{l,i-1} U_{\chi-1} V_{\chi} d\theta
\end{eqnarray*}
and 
\begin{eqnarray*}
E[| \int_0^1 X_l   Z_{l,i-1} T_{l,i} d\theta|^p] \leq c_p  \sum_{\chi=1}^{2^i} E[| \int_0^1 X_l Z_{l,i-1} U_{\chi-1} V_{\chi} d\theta|^p]
\end{eqnarray*}

Now $|V_\chi | = |e^{a_{I_\chi}\psi_{I_\chi}(\theta) - \frac{t_{k+1}}{2}\psi^2_{I_\chi}(\theta)}-1|\leq e^{|a_{I_\chi}\psi_{I_\chi}(\theta) - \frac{t_{k+1}}{2}\psi^2_{I_\chi}(\theta)|} |a_{I_\chi} - \frac{t_{k+1}}{2}\psi_{I_\chi}(\theta)| |\psi_{I_\chi}(\theta)|$ by the mean value theorem. We know that $|\psi_{I_\chi}| \leq c$ and with the obvious notation we have $|V_\chi | \leq e^{cW_\chi}W_\chi |\psi_{I_\chi}|$. It is easy to see that $E[e^{pcW_\chi}W_\chi^p] < C$ some universal constant.

So we have 
\begin{eqnarray*}
E[| \int_0^1 X_l Z_{l,i-1} U_{\chi-1} V_{\chi} d\theta|^p] \leq E[|\int_0^1 X_l Z_{l,i-1} U_{\chi-1} e^{cW_\chi}W_\chi |\psi_{I_\chi}(\theta)| d\theta|^p]\\
\leq   \left(\int_0^1 |\psi_{I_{\chi}}| d\theta\right)^p E[\left( \frac{1}{\int_0^1|\psi_{I_{\chi}}|  d\theta}\int_0^1 X_l Z_{l,i-1} U_{\chi-1} e^{cW_\chi}W_\chi|\psi_{I_{\chi}}| d\theta\right)^p]  \\
 \leq   \left(\int_0^1 |\psi_{I_{\chi}}| d\theta\right)^p E[ \frac{1}{\int_0^1|\psi_{I_{\chi}}|  d\theta}\int_0^1 X_l^p Z_{l,i-1}^p U_{\chi-1}^p e^{pcW_\chi}W_\chi^p|\psi_{I_{\chi}}| d\theta]\\
 \leq \left(\int_0^1 |\psi_{I_{\chi}}| d\theta\right)^p \frac{1}{\int_0^1|\psi_{I_{\chi}}|  d\theta}\int_0^1E[X_l^p Z_{l,i-1}^p U_{\chi-1}^p e^{pcW_\chi}W_\chi^p]|\psi_{I_{\chi}}| d\theta \\
 =  \left(\int_0^1 |\psi_{I_{\chi}}| d\theta\right)^p \frac{1}{\int_0^1|\psi_{I_{\chi}}|  d\theta}\int_0^1E[X_l^p] E[Z_{l,i-1}^p] E[U_{\chi-1}^p] E[e^{pcW_\chi}W_\chi^p]|\psi_{I_{\chi}}| d\theta 
 \end{eqnarray*}
 The last inequality holds by Jensen's inequality, while the equality holds because the random variables $X_l^p, Z_{l,i-1}^p, U_{\chi-1}^p, e^{pcW_\chi}W_\chi^p$ are independent. 
 
 One may easily see that :
 \begin{eqnarray*}
 \int_0^1 |\psi_{I_{\chi}}|d\theta\leq C_1 |I_{\chi}| \\
 E[e^{pcW_\chi}W_\chi^p] < C, \mbox{ universal constant}\\
 E[U_{\chi-1}^p] = e^{\sum_1^{\chi-1}(p^2-p)\frac{t_{k+1}\psi_I^2(\theta)}{2}} \\
 E[Z_{l,i-1}^p] = e^{\sum_{I\in J(i-1)}(p^2-p)\frac{t_{k+1}\psi_I^2(\theta)}{2}}\\
 = e^{\sum_{I\in J_l(i-1)}(p^2-p)\frac{t_{k+1}\psi_I^2(\theta)}{2}}\\
 E[X_l^p] = E[e^{pS_k}] e^{\sum_{j=0}^{l-1}\sum_{I\subset J_j}(p^2-p)\frac{t_{k+1}\psi_I^2(\theta)}{2}} \\
 = e^{\sum_{|I|\geq 2^{-n_k}}(p^2-p)\frac{t_{I}\psi_I^2(\theta)}{2}}e^{\sum_{j=0}^{l-1}\sum_{I\subset J_j}(p^2-p)\frac{t_{k+1}\psi_I^2(\theta)}{2}}
 \end{eqnarray*}
 Putting all this information together we get:
 \begin{eqnarray*}
 E[X_l^p] E[Z_{l,i-1}^p] E[U_{\chi-1}^p] E[e^{pcW_\chi}W_\chi^p] \\
\leq C e^{\sum_{|I|\geq 2^{-n_k}}(p^2-p)\frac{t_{I}\psi_I^2(\theta)}{2}}e^{\sum_{|I|\in [2^{-n_k-i},2^{-n_k-1}]}(p^2-p)\frac{t_{k+1}\psi_I^2(\theta)}{2}} \\
 \leq C  e^{\sum_{|I|\geq 2^{-n_k}}(p^2-p)\frac{t_{I}\psi_I^2(\theta)}{2}}e^{(p^2-p)\frac{t_{k+1}}{2}\sum_{|I|\in [2^{-n_k-i},2^{-n_k-1}]}\psi_I^2(\theta)}
  \end{eqnarray*}
 for all $l,i,\chi$ in their respective ranges. 

We apply inequality $(\ref{VagueletInequality})$ to get
 \begin{eqnarray*}
 e^{(p^2-p)\frac{t_{k+1}}{2}\sum_{|I|\in [2^{-n_k-i},2^{-n_k-1}]}\psi_I^2(\theta)} \leq Ce^{(p^2-p)i\ln2\frac{t_{k+1}}{2}}
 \end{eqnarray*}
 Similarly,
 \begin{eqnarray*}
e^{\sum_{|I|\geq 2^{-n_k}}(p^2-p)\frac{t_{I}\psi_I^2(\theta)}{2}} \leq Ce^{(p^2-p)\sum_{m=1}^{k}\frac{(n_{m}-n_{m-1})t_m\ln2}{2}}
 \end{eqnarray*}

Finally we may write
\begin{eqnarray*}
E[| \int_0^1 X_l Z_{l,i-1} U_{\chi-1} V_{\chi} d\theta|^p] \leq  C e^{(p^2-p)\sum_{a=1}^{k}\frac{(n_{a}-n_{a-1})t_a\ln2}{2}} e^{C_0(p^2-p)i\ln2\frac{t_{k+1}}{2}}
\end{eqnarray*} 

Taking all this into consideration we get:
\begin{eqnarray*}
E[|F_{k+1}-F_k|^p] \leq c_p^3 CC_1^p\sum_{l=0}^{2^{n_k+1}} \sum_{i=1}^{n_{k+1}-n_k}\sum_{\chi=1}^{2^i}2^{-(n_k+i)p} 2^{(p^2-p)\sum_{m=1}^{k}\frac{(n_{m}-n_{m-1})t_m}{2}}2^{(p^2-p)i\frac{t_{k+1}}{2}} \\
\leq c_p^3 CC_1^p2^{n_k}2^{-pn_k}2^{(p^2-p)\sum_{a=1}^{k}\frac{(n_{a}-n_{a-1})t_a}{2}} \sum_{i=1}^{n_{k+1}-n_k} 2^i 2^{-pi}2^{(p^2-p)i\frac{t_{k+1}}{2}} \\
\leq c_p^3 CC_1^p 2^{-n_k(p-1)+ (p^2-p)\sum_{a=1}^{k}\frac{(n_{a}-n_{a-1})t_a}{2}} \sum_{i=1}^{n_{k+1}-n_k} 2^{-i(p-1) + (p^2-p)i\frac{t_{k+1}}{2}} \\
\leq c_p^3 CC_1^p 2^{-n_k(p-1)(1 - \frac{ p t_k}{2})} \sum_{i=1}^{n_{k+1}-n_k} 2^{-i(p-1)(1 - \frac{p t_{k+1}}{2})}
\end{eqnarray*}

We have used the fact that $$\sum_{m=1}^{k}\frac{(n_{m}-n_{m-1})t_m}{2} = \frac{n_kt_k}{2} + \sum_{m=1}^{k-2}\frac{n_m(t_m-t_{m+1})}{2}\leq \frac{n_kt_k}{2}$$
which follows from the fact that the sequence $\{t_k\}$ is increasing.
 
 We make two observations. First, notice that $t_{c} = 2$. Secondly, the first martingale difference, $F_1-F_0$, dominates all others and satisfies 
\begin{equation}\label{LpDominatingEstimate}
E[|F_1-F_0|^{p_1}]\leq \frac{C}{2^{(p_1-1)(1-\frac{p_1t_1}{2})}-1}
\end{equation}
The power $p_1$ is chosen such that $(p_1-1)(1-\frac{p_1t_1}{2}) > 0$. The closer $t_1$ is to $t_c$, the closer $p_1$ is to zero and the worse this bound is. 

The next step is to obtain bounds on $E[F_k \log(e+F_k)]$. 

By the mean value theorem (applied to $f(x) = x\ln (e+ x)$) we may write
\begin{eqnarray*}
E[|F_{k+1}\ln(e + F_{k+1}) - F_k \ln(e+ F_k)|] \\ 
\leq E[|F_{k+1}-F_k| \sup_{\xi\in(F_k, F_{k+1})}(\ln (e+\xi)+1)]\\
\leq E[|F_{k+1}-F_k| \ln (e+\sup_{ i\leq k+1}F_i)] + E[|F_{k+1}-F_k|]
\end{eqnarray*}
We need to bound the first term. We start by applying Holder inequality, and continue on the third step by Doob's inequality:
\begin{eqnarray*}
E[|F_{k+1}-F_k| \ln (e+\sup_{ i\leq k+1}F_i)] \leq E[|F_{k+1}-F_k|^p]^{1/p}E[\ln (e+\sup_{ i\leq k+1}F_i)^q]^{1/q}\\
E[\ln (e+\sup_{ i\leq k+1}F_i)^q] = 1+q\int_1^{\infty} \lambda^{q-1}P(\ln (e+\sup_{ i\leq k+1}F_i) >\lambda)d\lambda \\
P(\ln (e+\sup_{ i\leq k+1}F_i) >\lambda) = P(\sup_{ i\leq k+1}F_i > e^{\lambda}-e)\leq \frac{E[F_{k+1}]}{e^{\lambda}-e} \approx e^{-\lambda} \\
E[\ln (e+\sup_{ i\leq k+1}F_i)^q] = 1+q\int_1^{\infty} \lambda^{q-1}e^{-\lambda}d\lambda\leq Cq (q!)\\
E[|F_{k+1}-F_k| \ln (e+\sup_{ i\leq k+1}F_i)] \leq C E[|F_{k+1}-F_k|^p]^{1/p} (q (q!))^{1/q}\\  
\leq C E[|F_{k+1}-F_k|^p]^{1/p} q = C E[|F_{k+1}-F_k|^p]^{1/p} \frac{p}{p-1}.
\end{eqnarray*}
Now we put everything together ($C$ is universal, i.e. it doesn't depend on $k,p$). 
\begin{eqnarray*}
E[F_{k+1}\ln(e+ F_{k+1})]  \leq  E[F_k\ln(e+ F_k)] +E[|F_k-F_{k+1}|] +\\
+ C E[|F_{k+1}-F_k|^p]^{1/p} \frac{p}{p-1}\\ 
\end{eqnarray*}
We want to make sure that the second and third term add up to less than $s = \frac{2}{(k+1)^2}$. We use the estimates above for $p = p_{k+1}$, which we specify below. It depends on our choice $t_k$. 
\begin{eqnarray*}
E[|F_k-F_{k+1}|^{p_{k+1}}] \leq  c_{p_{k+1}} C2^{n_k}2^{(p_{k+1}^2-p_{k+1})\sum_{i=1}^{k}\frac{(n_{i}-n_{i-1})t_i}{2}} 2^{-n_kp_{k+1}}  = \\
= c_{p_{k+1}} C 2^{n_k(1-p_{k+1})+(p_{k+1}^2-p_{k+1}) S_k}
\end{eqnarray*}
where we have denoted $S_k = \sum_{i=1}^{k}\frac{(n_{i}-n_{i-1})t_i}{2}$. This implies:
\begin{eqnarray*}
\left(E[|F_k-F_{k+1}|^{p_{k+1}}]\right) ^{1/p_{k+1}}\leq C 2^{n_k\frac{1-p_{k+1}}{p_{k+1}}+(p_{k+1}-1) S_k +\frac{\ln c_{p_{k+1}}}{p_{k+1}}}
\end{eqnarray*}
For $t_k = t_c - k^{-\gamma}$ we take $p_k = 1 + k^{-\gamma}/2$. It suffices to have $n_k$  such that:
\begin{eqnarray*}
2^{n_k\frac{1-p_{k+1}}{p_{k+1}}+(p_{k+1}-1) S_k +\frac{\ln c_{p_{k+1}}}{p_{k+1}} -2\ln(p_{k+1}-1)} < \frac{1}{(k+1)^2} 
\end{eqnarray*}
which is equivalent to: 
\begin{eqnarray*}
n_k\frac{1-p_{k+1}}{p_{k+1}}+(p_{k+1}-1) S_k +\frac{\ln c_{p_{k+1}}}{p_{k+1}} -2\ln(p_{k+1}-1) < -2\ln(k+1) 
\end{eqnarray*}
A closer look at $S_k$ reveals that this is less than $n_kt_k/2$ so it suffices to find $n_k$ such that:
\begin{eqnarray*}
n_k\frac{1-p_{k+1}}{p_{k+1}}+(p_{k+1}-1) \frac{n_kt_k}{2} +\frac{\ln c_{p_{k+1}}}{p_{k+1}} -2\ln(p_{k+1}-1) < -2\ln(k+1)  \\
n_k(p_{k+1}-1)\left(\frac{t_k}{2}-\frac{1}{p_{k+1}}\right) <  -\frac{\ln c_{p_{k+1}}}{p_{k+1}} + 2\ln(p_{k+1}-1) -2\ln(k+1) 
\end{eqnarray*}
Which is equivalent to
\begin{eqnarray*}
n_k > \frac{-\frac{\ln c_{p_{k+1}}}{p_{k+1}} + 2\ln(p_{k+1}-1) -2\ln(k+1) }{(p_{k+1}-1)(\frac{t_k}{2}-\frac{1}{p_{k+1}})} \\
n_k > 2\frac{\ln c_{p_{k+1}} - 2p_{k+1}\ln(p_{k+1}-1) + 2p_{k+1}\ln(k+1)}{(p_{k+1}-1)(2 - p_{k+1}t_k)}.
\end{eqnarray*}
Keeping in mind that $c_{p_{k+1}}\approx \frac{1}{(p_{k+1}-1)^2}$(the constant that gives comparability between the martingale and square function, squared) we see that our initial choice of $n_k$ satisfies this requirement.
This concludes the proof of the $L\log L$ integrability of the martingale $F_k$.
\end{proof}

\subsection{Properties}

We now list the main properties of the random measure $\nu = \lim_k \nu_k$ (see (\ref{definitionmeasure}) for the definition).
\begin{theorem}\label{MeasureProperties}
The limiting measure satisfies:
\begin{itemize}
\item[(a)] Almost surely, for all intervals $I$, $\nu(I)>0$.
\item[(b)] Almost surely, for all intervals $I$ we have $\nu(I)\leq |I|^{a_k}$, where $a_k$ corresponds to the $t_k$ for which $|I|\in[2^{-n_k}, 2^{-n_{k-1}})$. In particular $\nu(I)\leq e^{-\sqrt{\log\frac{1}{|I|}}}$.
\item[(c)] For any subinterval $I$ the random variable $\nu(I)$ has all negative moments.
\end{itemize}

In particular, this measure is non-atomic and non-zero on any interval. 
\end{theorem}

\begin{proof} We begin by arguing that the measure is a.s. non-zero on each interval. Consider an interval $I$. The event $\nu(I) = 0$ is  independent of the behavior of any finite number of levels in the GFF and hence a tail event. Since the martingale $\nu_k(I)$ is uniformly integrable we must have $P(\nu_k(I) = 0) = 0$ by Kolmogorov's zero-one law. Thus the measure $\nu$ is almost surely non-zero on the interval $I$. Since there is a countable number of dyadic intervals we get that almost surely the measure is non-zero on any such interval $I$. As any other interval contains a dyadic interval we get $(a)$. 

We prove $(b)$ by first showing that given a dyadic interval $I$ with $|I| = 2^{-n}$ and $n\in[2^{-n_k}, 2^{-n_{k-1}})$ , then $$ P(\nu(I)>|I|^{a_k})\leq |I|^{1+\epsilon_k}.$$

We have
\begin{eqnarray*}
P(\nu(I)>|I|^{a_k})\leq P(\nu_k(I)\geq |I|^{a_k}/2)+ \sum_{i= k+1}^{\infty} P(|\nu_{i+1}(I) - \nu_{i}(I)|\geq |I|^{a_k} 2^{k-i} )
\end{eqnarray*}
The terms in this sum can be bounded using the same estimates as in the proof of the uniform integrability. We obtain the following:
\begin{eqnarray*}
 P(\nu_k(I)\geq |I|^{a_k}/2)\leq 2^{p_k}|I|^{p_k- a_k p_k}2^{(p_k^2-p_k)\frac{n_kt_k}{2}} = 4 |I|^{\zeta_{p_k} - a_kp_k}
\end{eqnarray*}
We use the notation $\zeta_p := p - (p^2-p)\frac{t}{2}$. The power $p_k$ corresponds to the maximum variance, $t_k$, that appears in the definition of the measure $\nu_k$.

We can pick $a_k$ small enough to have $\zeta_{p_k} - a_kp_k = 1+ \epsilon_k$.

The other terms satisfy the following inequalities (by the same estimates as above):
\begin{eqnarray*}
 P(|\nu_{i+1}(I) - \nu_{i}(I)|\geq |I|^{a_k} 2^{k-i} )\leq  C_{p_i} 2^{(i-k)p_i}|I|^{p_i - a_k p_i} 2^{(n_i-n)(1-p_i)}2^{(p_i^2-p_i)\frac{n_it_i}{2}}.
\end{eqnarray*}
We now check that for each $i\geq k+1$ we have:
\begin{eqnarray*}
C_{p_i} 2^{(i-k)p_i}|I|^{p_i - a_k p_i} 2^{(n_i-n)(1-p_i)}2^{(p_i^2-p_i)\frac{n_it_i}{2}} \leq 2^{i-k} 2^{-n_k(1+\epsilon_k)}
\end{eqnarray*}
It suffices to have these inequalities for $n=n_k$:
\begin{eqnarray*}
C_{p_i} 2^{(i-k)p_i}2^{-n_k(p_i - a_k p_i)} 2^{(n_i-n_k)(1-p_i)}2^{(p_i^2-p_i)\frac{n_it_i}{2}} \leq 2^{i-k} 2^{-n_k(1+\epsilon_k)}
\end{eqnarray*}
This is equivalent to the following relation on $n_i$ ($i\geq k+1$):
\begin{eqnarray*}
(i-k)p_i -n_k(p_i-a_kp_i) +(n_i-n_k)(1-p_i) + (p_i^2-p_i)\frac{n_it_i}{2} \leq i-k -n_k(1+\epsilon_k)\\
i(p_i-1) + n_i(1-p_i) + (p_i^2-p_i)\frac{n_it_i}{2}\leq k(p_i-1) - n_k(\epsilon_k  + a_kp_i)
\end{eqnarray*}
For our choice of $n_i$ these relations are satisfied. To conclude the proof of part $(b)$ we consider the following (with $|I| = 2^{-n}$) 
\begin{eqnarray*}
\sum_{I} |I|^{1+\epsilon_k} = \sum_k\sum_{n=n_k}^{n_{k+1}-1} |I|^{-1}|I|^{1+\epsilon_k} \leq \sum_k 2^{-n_k\epsilon_k}\frac{1}{1-2^{-\epsilon_k}}<\infty 
\end{eqnarray*}
If we shift the dyadic grid by $1/3$ we get good estimates for all intervals and then Borel-Cantelli implies $(b)$.

We now turn our attention to the existence of negative moments (part $(c)$). We will address this by estimates on the Laplace transform of the measure $\nu$ in the vein of \cite{AJKS09}. We start by analyzing $\nu([0,1])$ and obtain estimates on the Laplace transform by finding a recurrence relation. 

Ideally, we would write 
\begin{eqnarray}
\nu([0,1])\geq B (M_1+M_2)
\end{eqnarray}
where $M_1$ and $M_2$ are independent and have the same law as $\nu([0,1])$, up to scaling. $\nu([1/4,3/8])$ and $\nu([5/8,3/4])$ would be candidates. 

Unfortunately, the two candidates are not independent and the law of $\nu([1/4,3/8])$ differs from that of $\nu([0,1])$. The differences are:
\begin{itemize}
\item The definition of $\nu([1/4,3/8])$ contains vaguelets on levels $\{0,1,2\}$ which are above (in dyadic tree sense) $[1/4,3/8]$.
\item The variances appearing in $\nu([1/4,3/8])$ are slightly different from the ones in the definition of $\nu([0,1])$ relative to the corresponding interval. 
\item Modulo the first two differences, the law of $\nu([1/4,3/8])$ differs from the law of $\nu([0,1])$, by a factor of $1/8$ due to the change of variable that maps $[1/4,3/8]$ to $[0,1]$.
\end{itemize}

We start from the two candidates and decouple them in such a way that we get to our goal. 

First consider $$X := \inf_{\theta\in[0,1]}\sum_{|J|\geq 2^{-2}}(a_J\psi_J(\theta) - \frac{t_J}{2}\psi_J^2(\theta)).$$
Then $\nu([1/4,3/8])\geq e^X \tilde{\nu}([1/4,3/8])$ and $\nu([5/8,3/4])\geq e^X \tilde{\nu}([5/8,3/4])$, where the tilde stands for the measures constructed starting only with dyadic intervals of size less than $2^{-3}$.
It is easy to see that $X$ satisfies the same distributional inequality as the one given in lemma \ref{GaussianFieldLemma}.
Secondly, denote $I_1 = [1/4,3/8]$ and $I_2 = [5/8,3/4]$ and consider the measure $\nu_i(I_i)$ formed by using only the vaguelets starting with dyadic level $3$ and corresponding to intervals $J\subset 3I_i$.
Define also
\begin{equation*}
X_i := \inf_{\theta\in I_i}\sum_{J\not\subset 3I_i,|J|\leq |I_i|}(a_J\psi_J(\theta) - \frac{t_J}{2}\psi_J^2(\theta)).
\end{equation*}
Then
$$\tilde{\nu}(I_i)\geq e^{X_i}\nu_i(I_i)$$

The random variable $X_i$ can be written as
\begin{equation*}
X_i = \inf_{\theta\in I_i}(F_i(\theta) - F_i(\theta_0))  +F_i(\theta_0) - \sup_{\theta\in I_i} \sum_{J\not\subset 3I_i,|J|\leq |I_i|}\frac{t_J}{2}\psi_J^2(\theta). 
\end{equation*}
where $ F_i(\theta_0) = \sum_{|J|\not\subset 3I_i}a_J\psi_J(\theta_0)$. This term is a Gaussian random variable with mean zero and variance $\sum_{|J|\not\subset 3I_i} \psi_J^2(\theta_0)$ which is bounded by a universal constant $C_0$ according to lemma \ref{VagueletLemma}. The last term is also bounded by $C_0$, while the first is controlled by lemma \ref{GaussianFieldLemma}.
Then
$$P(X_i\geq \centering \lambda)\leq c(1+\lambda)e^{-\lambda^2/2}$$

The random variables $\nu_i(I_i)$ are now independent, but they are formed only using the vaguelets corresponding to $3I_i$ and thus are quite different from $\nu([0,1])$. We make them more similar by the following argument. Let
\begin{equation*} 
Y_i := -\sup_{\theta\in I_i}\sum_{J\not\subset 3I_i, |J|\leq |I_i|}(b_{J,i}\psi_J(\theta) - \frac{t_J}{2}\psi_J^2(\theta)).
\end{equation*}
We take the random variables $b_{J,i}$ to be $N(0,t_J)$ and independent of $a_J$ and of one another. Thus $Y_1$ is independent of $Y_2$. This allows us to write:
$$\nu_i(I_i)\geq e^{Y_i}\nu_{i,0}(I_i)$$
where $\nu_{i,0}$ is defined by using all intervals $J\subset[0,1]$ and, when $J\not\subset 3I_i$, by using the random variables $b_{J,i}$.
The two random variables $\nu_{1,0}(I_1)$ and $\nu_{2,0}(I_2)$ and independent. 
It is immediate that $Y_i$ behave exactly like $X_i$:
$$P(Y_i\geq \centering \lambda)\leq c(1+\lambda)e^{-\lambda^2/2}$$
Finally, we consider 
$$Z_i := \inf_{\theta\in I_i}\sum_{k=1}^{\infty}\sum_{l=0}^2\sum_{|J|=2^{-n_k -l}}(\tilde{a}_{i,J}\psi_J(\theta) - \frac{t_k-t_{k-1}}{2}\psi_J^2(\theta))$$
where $a_{i,J}\sim N(0,t_k-t_{k-1})$.
The distributional properties of the field $Z_i$ are the same as for $X_i$ and $Y_i$ by using the same argument in lemma \ref{GaussianFieldLemma} and observing that the variances form a telescoping series.

We then have 
$$\nu_{i,0}(I_i)\geq e^{Z_i}\nu_{i,1}(I_i)$$
where $\nu_{i,1}(I_{i})$ are independent and have the same law as $\frac{1}{8}\nu([0,1])$.

All this work takes care of the differences we mentioned above. We are now in a position to write:
\begin{eqnarray}\label{NegativeMomentRecurrence}
M \geq e^Xe^{X_1}e^{Y_1}e^{Z_1} \frac{1}{8}M_1 +e^Xe^{X_2}e^{Y_2}e^{Z_2} \frac{1}{8}M_2
\end{eqnarray}
where $M,M_1,M_2$ have the law of $\nu([0,1])$ and the last two are independent of one another. 

Let $\mathcal{L}(s) =E[e^{-sM}] $ be the Laplace transform of $M$. The recurrence relation (\ref{NegativeMomentRecurrence}) implies 
\begin{equation}
\mathcal{L}(s^2)\leq \frac{C(X,X_i,Y_i,Z_i)}{s} +\mathcal{L}^2(s)
\end{equation}
where $C(X,X_i,Y_i,Z_i)$ is a constant which depends on the corresponding variables. However, since all the variables $X,X_i,Y_i,Z_i$ behave like Gaussians with mean zero and variance $<8$, we won't worry about them.
Setting $f(s) = \frac{c}{s^{1/2}} + \mathcal{L}(s)$ we get $f(s^2)\leq f^2(s)$.
We claim that there is an $s_0$ such that $f(s_0)\leq 1-\epsilon < 1$.

It suffices to find $s_0$ such that $\mathcal{L}(s_0)$ is less than 1.

We begin with some estimates (using notation $M = F_\infty$):
\begin{eqnarray}
P(M<1/4) \leq  P(F_1 < 1/2) + P(F_1>1/2, F_\infty <1/4)
\end{eqnarray}
The first term is bounded using (\ref{LpDominatingEstimate}) and lemma \ref{PrimitiveEstimate}:
\begin{equation}
P(F_1 < 1/2) \leq 1- \frac{1}{2^{\frac{p}{p-1}}E[|F_1|^p]^{{\frac{1}{p-1}}}}\leq 1-\left(\frac{2^{(p-1)(1-\frac{pt}{2})}-1}{C2^p}\right)^{\frac{1}{p-1}}
\end{equation}
where $p$ is the value which corresponds to the variance $t<t_c$ appearing in the definition of $F_1$. We denote the bound by $1-4\epsilon$.

The second term can be bounded by
\begin{eqnarray}
P(F_1>1/2, F_\infty <1/4) \leq \sum_{i=2}^{\infty} P(|F_i-F_{i-1}|>1/4^{i-1})
\end{eqnarray}
and this in turn is small in comparison to $P(F_1 > 1/2)$  by the proof of theorem \ref{LLogLMomentEstimate} and the choice of the $\{n_k\}$.

All this implies $P(M<1/4)\leq 1- 3\epsilon$.
We now pick $s_0$ sufficiently large to have $e^{-s_0/4}<\epsilon$ and $\frac{c}{s_0^{1/2}}<\epsilon$. Then we have $f(s_0)<1-\epsilon$. We iterate and get $f(s_0^{2^k})\leq (1-\epsilon)^{2^k}$ and by monotonicity $f(s)\leq Cs^{-\delta}$, where $C < s_0^\delta$ and $\delta\sim -\frac{\ln(1-\epsilon)}{\ln (s_0)}$.

It follows that 
\begin{eqnarray}
E[M^{-\delta/2}]\leq \frac{\delta}{2e}\int_0^\infty s^{\delta/2-1} E[e^{-sM}] ds\leq    s_0^\delta
\end{eqnarray}

We finish the proof as in \cite{AJKS09} by bootstrapping using the inequality between the geometric and arithmetic mean. We get 
\begin{equation}
E[M^{-2q}]\leq C_{neg}E[M^{-q}]^2
\end{equation}
The constant $C_{neg}$ comes from the negative $q$ moments of the variables $X,X_i,Y_i,Z_i$.  It is a universal constant. 

To bound the first negative moment we need to do this operation $\sim\log_2 \frac{1}{\delta}$ times. We get the bound:
\begin{equation}\label{NegativeMomentEstimate}
E[M^{-1}]\leq \left(C_{neg} s_0^\delta\right)^\frac{1}{\delta} = C_{neg}^{1/\delta} s_0
\end{equation}
We give a summary of the variables that are relevant for this inequality: 
\begin{equation}
\epsilon\sim \left(\frac{(p-1)(1-pt/2)\ln2}{C2^p} \right)^{\frac{1}{p-1}}, s_0\sim\frac{1}{\epsilon^2}, \delta \sim \frac{\ln(1-\epsilon)}{\ln\epsilon}
\end{equation}
If we construct the measure $\nu$ starting with variance $t_k = 2-k^{-\gamma}$ the bound reads: 
\begin{equation}
\epsilon\sim k^{-3\gamma k^\gamma}, s_0\sim  k^{6\gamma k^\gamma}, \delta \sim  \frac{k^{-3\gamma k^\gamma}}{k^\gamma}
\end{equation}
As we take $t_k\rightarrow t_c$ this bound blows up. This points to the fact that as we approach criticality the measures become more and more concentrated. 

To obtain the negative moments of $\nu(I)$ for some smaller interval one decouples the levels "above" $I$. These levels will form a centered Gaussian field which behaves like a Gaussian variable with variance $\sim \log_2 \frac{1}{|I|}$ and mean given by a sum of squares of vaguelets. For the decoupled part one applies the analysis here keeping in mind that the first variance has changed and hence the bounds have increased.  
\end{proof}

The next lemma appears in \cite{L05}, but we give it here for the sake of completeness.
\begin{lemma} \label{PrimitiveEstimate}
 Let $X$ be a positive random variable with $E[X] = 1$ and $E[X^p]<\infty$ for some $p>1$. Then
\begin{equation}
P(X>\frac{1}{2})^{p-1}\geq \frac{1}{2^pE[X^p]}
\end{equation}
\end{lemma}
\begin{proof}
The proof is the same as in \cite{L05}. It is immediate that $E[X;x\geq 1/2]\geq 1/2$. Then
\begin{eqnarray*}
E[X^p] &\geq&  E[X^p; X>1/2]\\
   &=& P(X>1/2) E[X^p|X>1/2]\\
   &\geq& P(X>1/2) E[X|X>1/2]^p\\
   &\geq& P(X>1/2) \frac{E[X;X>1/2]^p}{P(X>1/2)^p}\\
   &\geq& \frac{1}{2^pP(X>1/2)^{p-1}}
\end{eqnarray*}
\end{proof}

\begin{lemma}\label{GaussianFieldLemma}
Let $I$ be a dyadic interval of size $2^{-i}$ and consider the Gaussian field $ F = \sum_{J\not\subset 3I, |J|\leq |I|}a_{J}\psi_J(\theta) $ with $a_J\sim N(0,t_J)$. Then there exist universal constants $c$ and $C$ such that for fixed $\theta_0$
\begin{equation}
P(\sup_{\theta\in I}|F(\theta)-F(\theta_0)| >C u )\leq c(1+u)e^{-u^2/2}.
\end{equation}

\end{lemma}
\begin{proof}
The lemma is a consequance of the Borel-TIS inequality (see e.g. \cite{AJKS09}) for the field $F(\theta)-F(\theta_0)$.  We merely need to prove that $E[|F(\theta) - F(\theta')|^2]\leq L|\theta-\theta'|$ for $\theta, \theta' \subset I$. 
To this end observe that ($\theta^*\in [\theta, \theta']$ and depends on $J$):
\begin{eqnarray}
E[|F(\theta) - F(\theta')|^2] \leq \sum_{j=i}^\infty \sum_{|J|=2^{-j},J\not\subset 3I} t_J |\theta-\theta'||\psi_J'(\theta^*)|(|\psi_J(\theta)| + |\psi_J(\theta')|)\\
\leq C|\theta-\theta'| \sum_j \sum_{|J|=2^{-j},J\not\subset 3I} \frac{2^j}{(1+|2^j\theta^*-l|)^{q-1}} \frac{1}{(1+|2^j\theta-l|)^{q}}\\
\leq C |\theta-\theta'| \sum_j 2^j\sum_{|J|=2^{-j},J\not\subset 3I}  \frac{1}{(1+|2^jdist(J,I)|)^{2q-1}}\\
\leq  C |\theta-\theta'| \sum_j 2^j\sum_{d=2^{j-i}}^{2^j}  \frac{1}{(1+d)^{2q-1}}\leq C |\theta-\theta'| \sum_{j=i} \frac{2^j}{2^{(j-i)(2q-1)}}\\
\leq C|\theta-\theta'|2^i
\end{eqnarray}
since we have $q>1$. So $L= C2^i$. Since we are interested in the supremum over an interval of size $|I| = 2^{-i}$, Borel-TIS inequality gives the desired estimate.
\end{proof}

\section{Decoupling}
\label{SectionDecoupling}

In this chapter we prove that there are sequences $\rho_n,\tilde\rho_n, N_n, b_n, c_n$ such that:
\begin{equation*}
P(\sum_{i=1}^{N_n}Mod(G(A(z,\tilde\rho_n\rho_n^{i},2\tilde\rho_n\rho_n^{i}))) < c_n N_n)\leq \tilde\rho_n\rho_n^{(1+b_n)N_n} 
\end{equation*}
 for any $z\in \BBT$ and any mapping $G$ with Beltrami coefficient $\mu$.

Before we proceed we recall a few notations: in the following we will work with dyadic intervals $I,J$. $\mathcal{D}_n$ is the collection of dyadic intervals of size $2^{-n}$. For a dyadic interval $I$, $j(I)$ is the union of $I$ and it's two neighbors of the same size. We will extensively use the notation $C_I := \{(x,y)| x\in I, 2^{-n-1}\leq y \leq 2^{-n}\}$. These sets are called Whitney squares/boxes.

We now fix a point $z\in \BBT$ and $n$. For any $i\in{1,\ldots, N_n}$ define annuli $A_i :=A(z,\tilde\rho_n\rho_n^{i},2\tilde\rho_n\rho_n^{i})$, and balls $B_i := B(z, \tilde\rho_n\rho_n^{i})$, $B_i' := B(z, \tilde\rho_n\rho_n^{i}/4)$ . For simplicity of exposition, we will take $A_i$ to be square annuli. Because the picture is a local one, we think of the annuli as being centered at a point on $\BBR$.

The annulus $A_i $ can be divided in two parts:
\begin{itemize}
\item $R_{i,1} = A_i \cap \left(\cup_{I\in\mathcal{C}_i}C_I\right)$  , where $\mathcal{C}_i$ is the collection of all dyadic intervals $I$ of size at most $\tilde\rho_n\rho_n^i/4$ such that $I\cap A_i\neq \emptyset$.
\item $R_{i,2}=A_i\setminus R_{i,1}$
\end{itemize}

A big modulus $mod(G(A_i))$ is a consequence of controlled distortion in $A_i$. 

Distortion in any $C_I$ depends on the doubling properties of the random measure $\nu$ in $j(I)$. Let $I$ be a dyadic interval whose doubling properties affect the distortion in $A_{i}$ and any $\bold{J} = \{J_1,J_2\}$ with $J_i\subset j(I)$. 

For $k\in \{-1,0,\ldots, i-1\}$ define 
\begin{equation}
t_{i,k} := \sup_{8B_i} \sum_{ J: C_J\subset 16B_k, C_J\not\subset 16B_{k+1}}a_j\psi_J(\theta)- \inf_{8B_i}  \sum_{ J: C_J\subset 16B_k, C_J\not\subset 16B_{k+1}}a_j\psi_J(\theta)
\end{equation}
By lemma \ref{VagueletLemma} we have 
\begin{eqnarray*}
 \frac{\nu(J_1)}{\nu(J_2)} \leq C \frac{\nu_i(J_1)}{\nu_i(J_2)} e^{\sum_{k<i} t_{i,k}}
\end{eqnarray*}
where $\nu_i$ is the measure obtained only by using the dyadic intervals which are included in $16B_i$ and $C$ is a universal constant coming from the terms $e^{\sum_J t_J\psi^2_J/2}$.

The $t_{i,k}$ satisfy the following:
\begin{lemma}\label{DistribIneqt}
For any indices $i,i'$ the random variables $t_{i,k}$ and $t_{i',k'}$  are independent if $k\neq k'$, and satisfy
$$ 
P(t_{i,k}>Cu\sqrt{\rho_n^i \rho_n^{-(k+1)}} = C u \rho_n^{(i-k-1)/2}) \leq c(1+u)e^{-u^2/c},\  k\in \{-1,0,\ldots , i-1\}
$$
where $c,C$ are universal constants
\end{lemma}
\begin{proof} Consider the Gaussian random field:
$$X(\theta) := \sum_{|I|=16\tilde\rho_n\rho_n^{k+1}}^{16\tilde\rho_n\rho_n^k}a_I \psi_I(\theta) - \sum_{|I|=16\tilde\rho_n\rho_n^{k+1}}^{16\tilde\rho_n\rho_n^k}a_I\psi_I(
\theta_0)$$
By the mean value theorem $E[(X(\theta)-X(\theta'))^2] \leq Ct_I|\theta-\theta'|\sum_{|I|=16\tilde\rho_n\rho_n^{k+1}}^{16\tilde\rho_n\rho_n^k} |I|^{-1} (|\psi_I(\theta)|+|\psi_I(\theta')|)$. Using the properties of the vaguelets this is less than $Ct_I |\theta-\theta'|\tilde\rho_n^{-1}\rho_{k+1}^{-1}$. We look at this random variable on the interval $\BBR\cap 8B_i$ which has length $8\tilde\rho_n\rho_n^i$. By applying the Borell-TIS inequality we get the desired distributional inequality for $\sup X(\theta)$.

The same inequality holds for the $\sup$ of the field where we use $-a_I$ in stead of $a_I$. The $\theta_0$ terms cancel and we get the conclusion.
\end{proof}

\subsection{Distortion in $R_{i,1}$}
The region $R_{i,1}$ has two pieces $R_{i,1}^l$ and $R_{i,1}^r$. We describe the decoupling for the left piece; the decoupling for the right piece is analogous.

Let $\mathcal{C}_{i,i}^l$ be the colection of all dyadic intervals $I$ which do not appear in any of the $t_{i,k}$, $I\subset 16B_i\cap \BBR$ and which satisfy one of the following conditions:
\begin{itemize}
\item $|I|\geq \tilde\rho_n \rho_n^i/4$
\item if $|I| < \tilde\rho_n \rho_n^i/4$ then $I \cap B_{i+1}=\emptyset$ and $I$ does not intersect the $\tilde\rho_n\rho_n^i/4$- neighborbood of $R_{i,1}^l$.
\end{itemize}

Let $tr_{i,i}^l :=\sup(\sum_{J\in \mathcal{C}_{i,i}^l}a_J\psi_J(\theta)) - \inf(\sum_{J\in \mathcal{C}_{i,i}^l}a_J\psi_J(\theta))$, where the $\sup$ and $\inf$ are over $\theta \in R_{i,1}^l\cap \BBR$. These variables have the following property:

\begin{lemma}\label{DistribIneqtrii}
The random variables $tr_{i,i}$ and $tr_{i',i'}$  are independent if $i\neq i'$, and satisfy
$$ P(tr_{i,i}>Cu) \leq c(1+u)e^{-u^2/c},
$$
where $c,C$ are universal constants
\end{lemma}

For each $k > i$ let $\mathcal{C}_{i,k}$ be the collection of dyadic intervals of length at most $\tilde\rho_n\rho_n^k$ (on $\BBR$) in $B_k\setminus B_{k+1}$ to which we add the dyadic intervals $J$ that are inside $B_k$ and intersect $B_{k+1}$, but are not contained in it. Let $tr_{i,k} :=\sup(\sum_{J\in \mathcal{C}_{i,k}}a_J\psi_J(\theta)) - \inf(\sum_{J\in \mathcal{C}_{i,k}}a_J\psi_J(\theta))$, where the $\sup$ and $\inf$ are over $\theta \in R_{i,1}$.
We may write then
\begin{eqnarray*}
 \frac{\nu(J_1)}{\nu(J_2)} \leq C_0\frac{\nu_{i,0}(J_1)}{\nu_{i,0}(J_2)} e^{\sum_{k<i}t_{i,k}} e^{\sum_{k\geq i}tr_{i,k}}
\end{eqnarray*}
for each $\bold{J} = \{J_1,J_2\}$ which influences the distortion in the region $R_{i,1}^l$. The random measure $\nu_{i,0}$ is obtained by considering only the random variables $a_J$ corresponding to the $J$ which are subsets of the $\tilde\rho_n\rho_n^i/4$- neighborbood of $R_{i,1}^l\cap \BBR$.

The terms of the form $\sum_J t_J\psi_J^2(\theta)$ cancel in the quotient $ \frac{\nu(J_1)}{\nu(J_2)}$ because of lemma \ref{VagueletLemma}.

We should actually consider the two parts of $R_{i,1}$ separately in the computations that follow. However, their behavior is very similar so we allow ourselves to treat them as a unit. See figure \ref{FigureRegion1}.

The random variables $tr_{i,k}$ are supremums of Gaussian fields and they behave like Gaussian variables. 
\begin{lemma}\label{DistribIneqtrik}
The random variables $tr_{i,k}$ and $tr_{i',k'}$  are independent if $k\neq k'$, and satisfy
$$ P(tr_{i,k}>Cu\rho_n^{(k-i)(q-1)}) \leq c(1+u)e^{-u^2/c},\  k > i
$$
where $c,C$ are universal constants
\end{lemma}

\begin{figure}[htbp]
  \centering
    \includegraphics[width=0.8\textwidth]{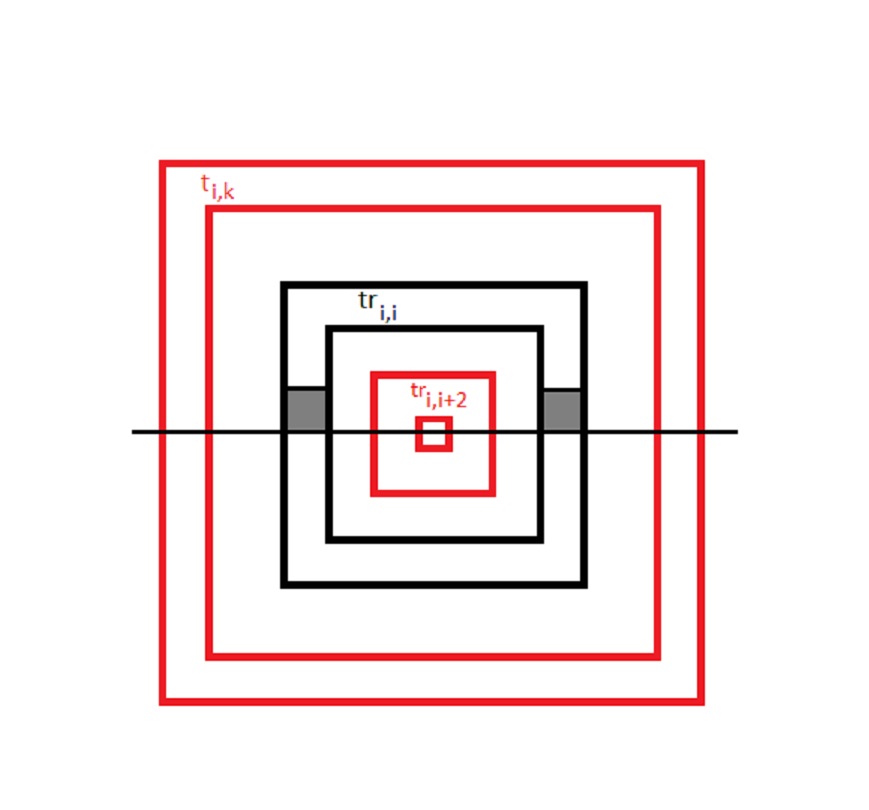}
\caption[Decoupling in $R_{i,1}$]{The decoupled variables used to control distortion in regions $R_{i,1}$(shaded)}
  \label{FigureRegion1}
\end{figure}

\begin{proof} The proof is very similar to the one of lemma \ref{GaussianFieldLemma}.

Consider the Gaussian random field: 
$$X(\theta) := \sum_{J\in \mathcal{C}_{i,k}}a_I\psi_I(\theta) - \sum_{J\in \mathcal{C}_{i,k}}a_I\psi_I(\theta_0) $$

Set $I = R_{i,1}^l\cap \BBR$. For any interval $J\in\mathcal{C}_{i,k}$ we have $dist(I,J)>\tilde\rho_n\rho_n^i/2$. We may now write ($\theta^*\in [\theta, \theta']$ and depends on $J$):
\begin{eqnarray}
E[|F(\theta) - F(\theta')|^2] \leq \sum_{j=-\log_2\tilde\rho_n\rho_n^k}^\infty \sum_{|J|=2^{-j},J\in\mathcal{C}_{i,k}} t_J |\theta-\theta'||\psi_J'(\theta^*)|(|\psi_J(\theta)| + |\psi_J(\theta')|)\\
\leq C|\theta-\theta'| \sum_j \sum_{|J|=2^{-j},J\in\mathcal{C}_{i,k}} \frac{2^j}{(1+|2^j\theta^*-l|)^{q-1}} \frac{1}{(1+|2^j\theta-l|)^{q}}\\
\leq C |\theta-\theta'| \sum_j 2^j\sum_{|J|=2^{-j},J\in\mathcal{C}_{i,k}}  \frac{1}{(1+|2^jdist(J,I)|)^{2q-1}}\\
\leq  C |\theta-\theta'| \sum_j 2^j\sum_{d=2^{j}\tilde\rho_n\rho_n^i/4}  \frac{1}{(1+d)^{2q-1}}\leq C |\theta-\theta'| \sum_{j=-\log_2\tilde\rho_n\rho_n^k} \frac{2^j}{(2^{j}\tilde\rho_n\rho_n^i)^{2q-1}}\\
\leq C|\theta-\theta'|\frac{(\tilde\rho_n\rho_n^k)^{2q-2}}{(\tilde\rho_n\rho_n^i)^{2q-1}}
\end{eqnarray}
since we have $q>1$. Since we are interested in the supremum over an interval of size $ \sim\tilde\rho_n\rho_n^i$, Borel-TIS inequality gives the desired estimate.

\end{proof}

The random measures $\nu_{i,0}$ are independent of one another. While we can not say that with high probability the distortion in the regions $R_{i,1}^l$ and $R_{i,1}^r$ is bounded (this was the case in \cite{AJKS09}), we will construct a stopping time region inside each of these where the distortion grows in a controlled fashion. 

\subsection{Distortion in $R_{i,2}$}

Let $\mathcal{I}_i$ be the set of $I\in\mathcal{D}$ such that $C_I$ intersects $A_i$ and $|I|\geq \tilde\rho_n\rho_n^i$. The distortion in $R_{i,2}$ is the same as the distortion in all $C_I$ for $I\in \mathcal{I}_i$. This is a finite (and universally bounded) number of intervals. For each of them we only need to control $\sim2^{10}$ pairs $\bold{J}$.  We have already decoupled the influence of variables $a_J$ with $J$ ouside of $16B_i$ on the measure $\nu$. For each pair $\bold{J}$ we can write (following \cite{AJKS09} and recalling that  $B_i' = B(z, \tilde\rho_n\rho_n^{i}/4)$)
\begin{eqnarray}
\delta_{\nu_i}(\bold{J}) = \delta_{\nu_i}(J_1\setminus B_i', J_2\setminus B_i') + \frac{\nu_i(J_1\cap B_i')}{\nu_i(J_2\setminus B_i')} +\frac{\nu_i(J_2\cap B_i')}{\nu_i(J_1\setminus B_i')}\\
\nu_i(J_j\cap B_i')=\sum_{k=i+1}^{N_n} \nu_i(J_j\cap B_{k-1}'\setminus B_{k}')
\end{eqnarray}
Define
\begin{eqnarray}
L_{i,i} = \sum_{(J_1,J_2)\in j(I), I\in \mathcal{I}_i}\delta_{\nu_i}(J_1\setminus B_i',J_2\setminus B_i')\\
L_{i,k} = \sum_{(J_1,J_2)\in j(I), I\in \mathcal{I}_i} \frac{\nu_i(J_1\cap (B_{k-1}'\setminus B_k'))}{\nu_i(J_2\setminus B_i')} + (1\leftrightarrow 2) \mbox{ for } i+1\leq k\leq N_i
\end{eqnarray}

An upper bound on $\sum_{k\geq i} L_{i,k}$ means the distortion in region $R_{i,2}$ is bounded. Before we give a distributional inequality for these random variables, we need to decouple them one more time. 
\begin{eqnarray}
 \frac{\nu_i(J_1\cap (B_{k-1}'\setminus B_k'))}{\nu_i(J_2\setminus B_i')} \leq  \frac{\nu_{i,k}(J_1\cap (B_{k-1}'\setminus B_k'))}{\nu_{i,k}(J_2\setminus B_i')} e^{\sup(\sum_J\ldots) - \inf(\sum_J\ldots)}
\end{eqnarray}
where measure $\nu_{i,k}$ is constructed using only the random variables $a_J$ for which $J\subset 16B_i$ and $J\not\subset \frac{1}{2}B_k'$. In addition, the $\sup$ and $\inf$ are considered over the set $\BBR\cap( 6B_i\setminus B_k')$. Denote these variables by $\sum_{j>k} s_{i,k,j}$, where $s_{i,j,k}$ involves only the dyadic intervals $J$ which are subsets of $\frac{1}{2}B_{j-1}'$, but are not subsets of $\frac{1}{2}B_{j}'$.

For simplicity of notation, we will use $L_{i,k}$ for the sums above, but involving $\nu_{i,k}$. Distortion will then be bounded by 
\begin{eqnarray}
\sum_{k\geq i} L_{i,k} e^{\sum_{j>k}s_{i,k,j}}
\end{eqnarray}
For a picture of the decoupled variables see Figure \ref{FigureRegion2}.

The random variables $s_{i,k,j}$ have similar properties as $t_{i,k}, tr_{i,k}$ above due to the same reason: they are supremums of Gaussian fields. 
\begin{lemma}\label{DistribIneqs}
The random variables $s_{i,k,j}$ and $s_{i',k', j'}$  are independent if $j\neq j'$, and satisfy
$$ P(s_{i,k,j}> u\rho_n^{(j-k)q/2}) \leq Ce^{-u^2/4},\  j > k >i
$$
where $C$ is a universal constant.
\end{lemma}

\begin{figure}[h]
   \centering
    \includegraphics[width=0.8\textwidth]{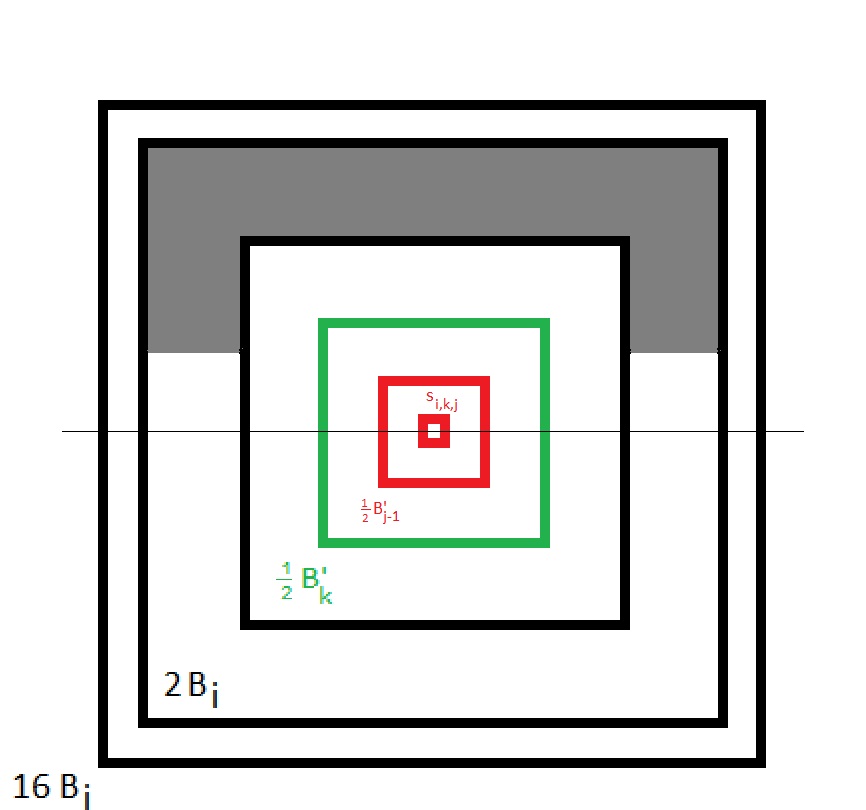}
\caption[Decoupling in $R_{i,2}$]{The decoupled variables $L_{i,k}$ control the distortion in the shaded region using the vaguelets from outside the green square. $s_{i,k,j}$ control distortion in $R_{i,2}$ using vaguelets inside the two red squares.}
  \label{FigureRegion2}
\end{figure}

\begin{proof} The usual argument involving the Borel-TIS lemma (see e.g. \cite{AJKS09}) doesn't give a good estimate, because the lemma deals with more general Gaussian fields. We will give the simplest possible argument. 

For simplicity of notation, let $2^{-l_1} = \tilde\rho_n\rho_n^k$ and $2^{-l_0} = \tilde\rho_n\rho_n^j$.
All the sums refer to intervals $J$ which appear in $s_{i,k,j}$ and $\theta$ is any point in $\BBR\cap( 6B_i\setminus B_k')$.

Then
\begin{eqnarray}
P(\sum_{J} a_J\psi_J(\theta) \geq \lambda)\leq \sum_{l=l_0}^\infty P(\sum_{|J|=2^{-l}}|a_J||\psi_J(\theta)|\geq \frac{\lambda}{2^{l-l_0}})\\
\leq \sum_{l=l_0}^\infty \sum_{|J|=2^{-l}} P(|a_J||\psi_J(\theta)|\geq \frac{\lambda}{2^{2(l-l_0)}})\\
\leq  \sum_{l=l_0}^\infty \sum_{|J|=2^{-l}} P(|a_J|\frac{C}{(2^ldist)^q)}\geq \frac{\lambda}{2^{2(l-l_0)}})\\
\leq  C \sum_{l=l_0}^\infty \sum_{|J|=2^{-l}} e^{-\frac{\lambda^2 2^{q(l-l_1)}}{4\ 2^{2(l-l_0)}}}\leq C\sum_{l=l_0}^\infty 2^{l-l_0}e^{-\frac{\lambda^2 2^{(q-2)(l-l_0) +q(l_0-l_1)}}{4}}
\end{eqnarray}
Here $C$ is a universal constant and $dist$ is the distance between the region where $J$ are and the region where we take the sup/inf over. The statements above should be understood for all $\theta\in\BBR\cap( 6B_i\setminus B_k')$, and not for a particular $\theta$.

Replace $\lambda$ by $\lambda 2^{-q(l_0-l_1)/2}$ and get :
\begin{eqnarray}
P(\sup_\theta \sum_{J} a_J\psi_J(\theta) \geq \lambda  2^{-q(l_0-l_1)/2})\leq C e^{-\lambda^2/4}
\end{eqnarray}
Recalling that $2^{-l_1} = \tilde\rho_n\rho_n^k$ and $2^{-l_0} = \tilde\rho_n\rho_n^j$ we get the conclusion.

\end{proof}

The random variables $L_{i,k}$ satisfy the following distributional inequality. 
\begin{lemma} \label{DistribIneqL}
There exists $a_n>0$ (depending only on the sequence of variances in the GFF)  and $C_n<\infty$ (independent of $i,k,\rho$) such that 
$$
 P(L_{i,k} > \lambda)\leq C_n \lambda^{-1} \rho_n^{(k-i-1)(1+a_n)}
$$
In addition, $L_{i, k}$ and $L_{j,l}$ are independent if $k<j$ or $l < i$.
\end{lemma}
\begin{proof}
We recall that $L_{i,k}$ is a sum of a finite (and universaly bounded) number of terms of the form $\frac{\nu_{i,k}(J_1\cap (B_{k-1}'\setminus B_k'))}{\nu_{i,k}(J_2\setminus B_i')}$. For simplicity of exposition we will redenote this quantity by $\frac{\nu(J)}{\nu(I)}$ where $I, J$ are two intervals of sizes $|I|\sim \tilde\rho_n\rho_n^{i}$ and $|J|\sim \tilde\rho_n\rho_n^{k-1}$.

The measure $\nu$ is constructed using only the random variables $a_J$ for which $J\subset 16B_i$ and $J\not\subset \frac{1}{2}B_k'$. The distributional properties of this measure are no different from the properties of the measure constructed using all the $J\subset 16B_i$. This is the case because the vaguelets corresponding to $\frac{1}{2}B_k'$ form a Gaussian field with a controlled variance when evaluated inside $16B_i\setminus B_k'$. We use the same notation, $\nu$, for this more "complete" measure.

We now have
\begin{eqnarray*}
P\left(\frac{\nu(J)}{\nu(I)}>\lambda\right) &\leq& P\left(\frac{\nu_2(J)}{\nu(I)}>\lambda /2 \right)+ P\left(\frac{\nu(J)-\nu_2(J)}{\nu(I)}>\lambda /2 \right) \\
&\leq&  P\left(\frac{\nu_2(J)}{\nu(I)}>\lambda /2 \right) + \sum_{i=3}^{\infty} P\left(\frac{\nu_{i}(J)-\nu_{i-1}(J)}{\nu(I)}>\lambda /2^{i-1} \right)
\end{eqnarray*}
Here the variables $\nu_i$ are the martingale approximations to $\nu$.
We have the following trivial inequality:
\begin{equation}
P\left(\frac{\nu_2(J)}{\nu(I)}>\lambda /2 \right) \leq \left(\frac{2}{\lambda}\right)^{\tilde q_2}E[\left(\frac{\nu_2(J)}{\nu(I)}\right)^{\tilde q_2}].
\end{equation}
Let $p_2>\tilde{q}_2$ for which $E[\nu_2(J)^{p_2}]<\infty$. Inspection of theorem \ref{LLogLMomentEstimate} reveals that:
\begin{eqnarray}
E[\nu_2(J)^{p_2}]\leq C_{p_2} |J|^{\zeta_{p_2}} |I|^{\frac{(p_2^2-p_2)t_2}{2}}
\end{eqnarray}
where $\zeta_{p_2} := p_2-\frac{(p_2^2-p_2)t_2}{2}$.

We know that $\nu(I)$ has negative moments of all orders. In particular, theorem \ref{MeasureProperties} gives us the following bound for the negative moment:
\begin{eqnarray}
E[\nu(I)^{-q_2\tilde{q}_2}]\leq C|I|^{-q_2\tilde{q}_2} \left(C_{neg}^{\frac{1}{\delta_1}} s_1\right)^{q_2\tilde{q}_2}
\end{eqnarray}
where $q_2 = \frac{p_2}{p_2-\tilde{q}_2}$ (the conjugate of $\frac{p_2}{\tilde{q}_2}$), and $\delta_1, s_1$ are the constants which appear in the proof of theorem \ref{MeasureProperties}.

Applying Hoelder inequality and combining the last few estimates we get:
\begin{eqnarray}
P\left(\frac{\nu_2(J)}{\nu(I)}>\lambda /2 \right) \leq C  C_{p_2}^{2} \left(\frac{2}{\lambda}\right)^{\tilde q_2} |J|^{\tilde{q}_2\zeta_{p_2}/p_2} |I|^{\frac{\tilde{q}_2}{p_2}\frac{(p_2^2-p_2)t_2}{2}}|I|^{-\tilde{q}_2} \left(2^{\frac{1}{\delta_1}} s_1\right)^{\tilde{q}_2}\\
\leq C(\tilde{q}_2,p_2,t_1)  \left(\frac{1}{\lambda}\right)^{\tilde q_2}\left(\frac{|J|}{|I|}\right)^{\tilde{q}_2\zeta_{p_2}/p_2}\\
\leq  C(\tilde{q}_2,p_2,t_1)  \frac{1}{\lambda}\left(\frac{|J|}{|I|}\right)^{1+a_2}
\end{eqnarray}
where $C(\tilde{q}_2,p_2,t_1) $ contains all the constants and $\tilde{q}_2\zeta_{p_2}/p_2 = 1+a_2$. This constant is dominated by $\left(2^{\frac{1}{\delta_1}} s_1\right)^{\tilde{q}_2}$ and $\delta_1,s_1$ depend on the first variance which appears in the definition of $\nu$ (hence the parameter $t_1$).

We deal with the terms of the form 
$$
P\left(\frac{\nu_{i+1}(J)-\nu_{i}(J)}{\nu(I)}>\lambda /2^{i+1} \right)
$$
in the analoguous fashion. To bound the numerator one chooses $p_{i+1}$ for which the corresponding moment of $\nu_{i+1}(J)-\nu_{i}(J)$ exists. The computation of this moment is basically given in theorem \ref{LLogLMomentEstimate}.
\begin{eqnarray}
E[|\nu_{i+1}(J)-\nu_i(J)|^{p_{i+1}}]\leq C_{i+1} |J|^{\zeta_{p_{i+1}}}|I|^{\frac{(p_{i+1}^2-p_{i+1})t_1}{2}}2^{- (n_i-j)(p_{i+1}-1)(1-p_{i+1}t_{i}/2)}
\end{eqnarray}
where $C_{i+1} = \frac{C}{2^{(p_{i+1}-1)(1-p_{i+1}t_{i+1}/2)}}$ and $|J| = 2^{-j}$.

Combined with the negative moment estimate this gives:
\begin{eqnarray}\nonumber
P\left(\frac{\nu_{i+1}(J)-\nu_{i}(J)}{\nu(I)}>\lambda /2^{i+1} \right)\leq C(\tilde{q}_{i+1},p_{i+1},t_1)  \left(\frac{2^{i+1}}{\lambda}\right)^{\tilde q_{i+1}}\frac{|J|}{|I|^{\tilde{q}_{i+1}\zeta_{p_{i+1}}/p_{i+1}}} 2^{-n_i (\zeta_{p_{i+1}}-1)\frac{\tilde{q}_{i+1}}{p_{i+1}}}
\end{eqnarray}

So we get (after we change the index)
\begin{eqnarray}
\sum_{l=2} P()\leq \frac{|J|}{\lambda|I|}\sum_{l=2}C(\tilde{q}_{l+1},p_{l+1},t_1) 2^{(l+1)\tilde{q}_{l+1}}\frac{1}{|I|^{\tilde{q}_{l+1}\zeta_{p_{l+1}}/p_{l+1}-1}} 2^{- n_l(\zeta_{p_{l+1}}-1)\frac{\tilde{q}_{i+1}}{p_{i+1}}}
\end{eqnarray}

Since $|I|\sim \tilde\rho_n\rho_n^{i}$ and $|J|\sim \tilde\rho_n\rho_n^{k-1}$ we have $\frac{|J|}{|I|}\sim\rho_n^{k-1-i}\geq \rho_n^{N_n-1-i} \geq \rho_n^{N_n}$. 

The variance $t_1$ corresponds to levels $\sim \tilde{\rho}_n\rho_n^{N_n}$ which in turn is larger than $2^{-n_1}$ in the notation of this lemma.

$|I|$ is between $\sim \tilde\rho_n\rho_n^{N_n}$ and $\tilde{\rho}_n$. To get the desired estimate it suffices to have:
\begin{eqnarray}
C(\tilde{q}_{3},p_{3},t_1) 2^{3\tilde{q}_{3}} 2^{- n_2(\zeta_{p_{3}}-1)\frac{\tilde{q}_{3}}{p_{3}}} \leq\frac{1}{2} \rho_n^{N_n}\left(\tilde{\rho}_n\rho_n^{a_2N_n}\right)^{(\tilde{q}_{3}\zeta_{p_{3}}/p_{3}-1)}
\end{eqnarray}
Since $2^{-n_1+n_0}\leq \rho_n^{N_n}$ ($n_0$ corresponds to the levels above $\tilde\rho_n$) and $2^{-n_1} < \tilde{\rho}_n\rho_n^{N_n}$ , it suffices to take:
\begin{eqnarray}
C(\tilde{q}_{3},p_{3},t_1) 2^{3\tilde{q}_{3}} 2^{- n_2(\zeta_{p_{3}}-1)\frac{\tilde{q}_{3}}{p_{3}}} \leq 2^{(-n_1+n_0)a_2-n_1 a_3}
\end{eqnarray}
If we also have the following relations on $n_k$:
\begin{eqnarray}
C(\tilde{q}_{k+2},p_{k+2},t_k) 2^{3\tilde{q}_{k+2}} 2^{- n_{k+1}(\zeta_{p_{k+2}}-1)\frac{\tilde{q}_{k+2}}{p_{k+2}}} \leq 2^{(-n_{k}+n_{k-1})a_{k+1}-n_k a_{k+2}}, \forall k>1\label{RecursionRelationDueToL}
\end{eqnarray}
we get the desired conclusion. This is because these relations, although not the same as the ones present in the sum, dominate the latter. 

In the case when $t_k = 2-k^{-\gamma}$ and $p_k = 1+ k^{-\gamma}/2$ we get $\zeta_{p_k} = 1+k^{-3\gamma}/8$. We take $q_k = p_k(1-k^{-3\gamma}/16)$ and get $a_k = k^{-3\gamma}/16 - k^{-6\gamma}/12\sim k^{-3\gamma}/16$.

Since $C(\tilde{q}_{k+2},p_{k+2},t_k)\leq C_{neg}^{k^{3\gamma k^\gamma}}k^{12\gamma k^\gamma}$ relation (\ref{RecursionRelationDueToL}) can be simplified as follows:
\begin{eqnarray}
 ck^{3\gamma k^\gamma} - n_{k+1}a_{k+2} \leq (-n_{k}+n_{k-1})a_{k+1}-n_k a_{k+2}
\end{eqnarray}
so it suffices to take $ n_{k+1}a_{k+2} \sim  c(k+1)^{3\gamma (k+1)^\gamma}$ or $n_{k+1} \sim (k+2)^{3\gamma} c(k+1)^{3\gamma (k+1)^\gamma}$.

\end{proof}

Recall that we are trying to prove an estimate on moduli of annuli at scales between $\tilde\rho_n\rho_n^{N_n}$ and $\tilde\rho_n$. All the work we have done decoupling the distortion was done to address a fixed $n$. The distributional inequality is satisfied by all $L_{i,k}$, where $i \leq k\leq N_n$. The constant $a_n$ depends on first variances in the definition on the measures $\nu_{i,k}$. For each $n$, all the first variances are equal to the same value $t_n$. As $n$ increases, $a_n\rightarrow 0$. This is a big difference between the critical case and the non critical case treated in \cite{AJKS09}. In the sub-critical case the variables $L_{i,k}$ were defined for all scales $k>i$ and they all shared the same value of the constant $a$.

\section{Random tree}\label{SectionStoppingTime}

We need to control the distortion in the regions $R_{i,2}$. The authors of \cite{AJKS09} were able to get bounded distortion with high probability. In our case this is not possible anymore essentially because as the variances $t_n\rightarrow t_c$ we lose control over $L_{i,k}$. We use a stopping time algorithm to construct a random tree. We devise a collection of rules which we apply to dyadic squares. If all these rules are satisfied for a particular interval/square, then the distortion will be controlled. Our goal is to obtain an infinite d-ary surviving tree where the distortion is controlled.

\subsection{The rules that define the tree}

In the following $A, A_{t}$, $B$ are large constants, $\delta$ a small constant and $N$ is a large positive integer.

We construct stopping rules on a tree in which each node represents an interval of size $2^{-i(N-4)}$. For each interval $I$, a node in this tree, we denote by $Ans_n(I)$ the ancestor $n$ levels above (and of size $2^{-(i-n)(N-4)}$).

We start by considering a dyadic interval $I$ of size $2^{-i(N-4)}$ and measure $\nu$ which is constructed only using the vaguelets corresponding to intervals $J$ which are subsets of $I$ and its closest four neighbors (two to the left, two to the right - call this set $\mathcal{J}(I)$).   Out of all the dyadic subintervals of $j(I)$ of size $|I|2^{-Nj}$ we select and mark half in an alternating fashion and call them $I_k^j$. 

The interval $I$ survives if all of the following good events take place:
\begin{itemize}
\item[1)] 
\begin{equation}
\frac{1}{A_{t}} |I_k^1|\leq \nu(I_k^1)\leq A_{t}|I_k^1|,\forall k
\end{equation}
where $t$ represents the variance (in the definition of the GFF) for which all levels with that variance are below the level of $I_k^j$.
\item[2)] For each  $J\in\{I, I_l, I_r\}$ ( l and r stand for left and right neighbours)
\begin{eqnarray}
\sum_{I_k^j\subset J} \nu(I_k^j)\leq A |J| 2^{-j\delta}, \forall j>1
\end{eqnarray}
\item [3)] For each $n\leq i$:
\begin{eqnarray}
\sup_{\theta\in \mathcal{J}(I)}\sum_{J\subset \mathcal{J}(Ans_n((I))\setminus\mathcal{J}(I)}a_J\psi_J(\theta)  - a_J\psi_J(\theta_I)  \leq B 2^{-n}\\
\inf_{\theta\in \mathcal{J}(I)}\sum_{J\subset \mathcal{J}(Ans_n((I))\setminus\mathcal{J}(I)}a_J\psi_J(\theta)  - a_J\psi_J(\theta_I)  \geq -B 2^{-n}
\end{eqnarray}
where $\theta_I$ is the center of $I$.
\item[4)] 
\begin{eqnarray}
\sup_{\theta\in j(I)}\sum_{J\subset j(I),|J|\geq|I_k^1|2^5}a_J\psi_J(\theta)  - a_J\psi_J(\theta_I) \leq B\\
\inf_{\theta\in j(I)}\sum_{J\subset j(I),|J|\geq|I_k^1|2^5}a_J\psi_J(\theta)  - a_J\psi_J(\theta_I) \geq -B
\end{eqnarray}
where $\theta_I$ is the center of $I$.
\end{itemize}

All the constants are chosen such that 
\begin{equation}\label{probabilityRulesHold}
P(\mbox{rules hold})\approx 1
\end{equation} 
The constants $A_{t}$ vary with the level at which we apply the rules. As we apply these rules to smaller and smaller intervals $I$, the properties of the measure $\nu$ become weaker and weaker. We want to make sure that the probability (\ref{probabilityRulesHold}) doesn't change as we go deeper and deeper. 

If all these rules are satisfied the distortion inside $j(I)$ between heights $|I_k^1|2^5$ and $2^{-n}=2^{-i(N-4)} = |I|$ is bounded by $D_n$. We reiterate the fact that this sequence won't be bounded.

\begin{figure}[h!]
 \centering
    \includegraphics{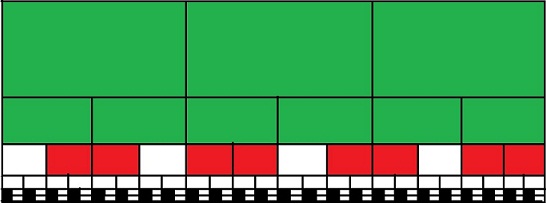}
    \caption[Construction of the random tree]{The interval in the middle is $I$. The black intervals are the marked $I_k^1$. The green boxes represent the area where the distortion is under control. At the next step we only look at the white boxes that lie between four red boxes(here shown as two due to lack of space).  }
    \label{StoppingTimeRegion}
\end{figure}

If all the rules are satisfied for $I$, we then look at its "children". We consider all the dyadic intervals of size $|I_k^1|2^4$ inside the region $j(I)$. We don't have control over the distortion in the boxes corresponding to these intervals. We select one third of these intervals in such a way that any two selected intervals are separated by four boxes which we do not select. See Figure \ref{StoppingTimeRegion}. There will be a total of $2^{N-4}/5$ such intervals. We now run the rules for each one of these intervals. \textit{When doing this, we consider the measure $\nu$ constructed only using the vaguelets corresponding to $\mathcal{J}(\mbox{this interval})$}. As we run these rules for smaller and smaller intervals we obtain a random tree.

\subsection{A surviving d-ary subtree with high probability}

The goal of this section is to provide estimates on the probability that each rule fails and then to obtain an estimate on the probability that there is a d-ary surviving subtree. 

We have
\begin{lemma}
$$
P(\mbox{rule 1 fails})\leq \frac{2}{A}
$$
\end{lemma}
\begin{proof}
It is immediate that 
\begin{equation}
P(\nu(I_k^1)\geq A_{t}|I_k^1|)\leq \frac{|I_k^1|}{|I_k^1|A_{t}} = \frac{1}{A_t}
\end{equation}
For the other inequality we need to use the negative moment estimate (\ref{NegativeMomentEstimate}). We first write $\nu(I_k^1)$ as $e^{\sum} |I_k^1|\tilde\nu([0,1])$ where $e^{\sum}$ contains the vaguelets from levels above that of $I_k^1$ and below that of $I$. Strictly speaking, $\tilde\nu$ is slightly different from the original construction, but there is no difference relevant for out computation.

This implies
\begin{eqnarray}
P(\frac{1}{A_{t}} |I_k^1|\geq \nu(I_k^1)) \leq \frac{C_{neg}^{1/\delta} s_0E[e^{-\sum}]}{A_t} = \frac{C_{neg}^{1/\delta} s_0e^{N}}{A_t}
\end{eqnarray}

We recall that the constants satisfy: 
\begin{equation}
\epsilon\sim \left(\frac{(p-1)(1-pt/2)\ln2}{C2^p} \right)^{\frac{1}{p-1}}, s_0\sim\frac{1}{\epsilon^2}, \delta \sim \frac{\ln\epsilon}{\ln(1-\epsilon)}
\end{equation}
where $t=t_k$ for which all the $n_k$ levels with $t=t_k$ are completely below level $2^{-i(N-4)+N}$. $p$ is the appropriate power.

If we set $A_{t_k} =A C_{neg}^{1/\delta_k} s_{0,k}e^{N}$ we get the conclusion.

\end{proof}
 
\begin{lemma}
$$
P(\mbox{rule 2 fails})\leq \frac{C}{A}
$$
\end{lemma}
\begin{proof}
We start by obtaining an estimate for a particular $j$. We remark that 
$$
\sum_{I_k^j\subset J} \nu(I_k^j)  = \nu(\cup_{I_k^j\subset J} I_k^j)
$$
and that the latter is simply the limit of a martingale $G_m =  \nu_m(\cup_{I_k^j\subset J} I_k^j)$ constructed in the same way as the original $F_k$. There are only a few differences:
\begin{itemize}
\item The measure $\nu$ is constructed only using vaguelets $\psi_*$ corresponding to levels starting at $2^{-i} = |J|$ and inside $\mathcal{J}(I)$.
\item The index $m = 0$ stands for $d\nu_0 = d\theta$, while $\nu_1$ stands for the measure obtained using all Gaussian random variables and vaguelets having the variance of level $2^{-(i(N-4)+jN)}$ (the level corresponding to $I_k^j$ ), denoted in this proof by $t_1$. 
\end{itemize}

We will obtain estimates on the following
\begin{eqnarray}
P(|G_{m+1}-G_m|> \frac{6A|J|2^{-j\delta}}{m^2\pi^2})
\end{eqnarray}
A look at the argument in theorem \ref{LLogLMomentEstimate} reveals that the same argument works in this case. The only differences come from the fact that we have to replace the interval $[0,1]$ by $S=\cup_{I_k^j\subset J} I_k^j$ (and the fact that the vaguelets decay fast enough).

We get the estimate:
\begin{eqnarray}
E[|G_{m+1}-G_m|^p] \leq
 C C2^{-ip} 2^{(N-1)j}2^{-Njp} 2^{(p^2-p)\frac{jNt_1}{2}} 2^{-n_m(p-1)(1 - \frac{ p t_m}{2})} 
\end{eqnarray}
The $n_m, t_m$ are all relative to the dyadic level of $S$.

We emphasize that the $p$ corresponds to the $t_m$, so the estimate should read:
\begin{eqnarray}
E[|G_{m+1}-G_m|^{p_m}]\leq C2^{-ip_m} 2^{(N-1)j}2^{-Njp_m}2^{(p_m^2-p_m)\frac{jNt_1}{2}} 2^{-n_m(p_m-1)(1 - \frac{ p_m t_m}{2})}
\end{eqnarray}

This immediately leads us to:
\begin{eqnarray}
P(|G_{m+1}-G_m|> \frac{6A|J|2^{-j\delta}}{m^2\pi^2}) \leq \\
\leq   C 2^{-ip_m} 2^{(N-1)j}2^{-Njp_m}2^{(p_m^2-p_m)\frac{jNt_1}{2}} 2^{-n_m(p_m-1)(1 - \frac{ p_m t_m}{2})}m^{2p_m} 2^{jp_m\delta}2^{p_m i}A^{-p_m}\\
=  C  2^{-j(1-\delta p_m)}2^{-Nj(p_m-1)+(p_m^2-p_m)\frac{jNt_1}{2}} 2^{-n_m(p_m-1)(1 - \frac{ p_m t_m}{2})}m^{2p_m} A^{-p_m}
\end{eqnarray}
This immediately implies:
\begin{eqnarray}
P(\sum_{I_k^j\subset J} \nu(I_k^j)> A |J| 2^{-j\delta})\leq C  2^{-j(1-\delta p_m)} A^{-1}
\end{eqnarray}
which in turn gives us 
\begin{equation}
P(\mbox{rule 2 fails})\leq CA^{-1}
\end{equation}
\end{proof}

\begin{lemma}
$$
P(\mbox{rule 4 fails})\leq C(1+CB) e^{-CB^2/2}
$$
\end{lemma}
\begin{proof}
Rule 4 is a condition about a centered Gaussian Field with covariance bounded by $C|I_k^j|^{-1}$. We are interested in its supremum/infimum over an interval of size $3|I|$, so by Borel-TIS these behave like Gaussian variables with mean zero and standard deviation $C\sqrt{|I_k^j|^{-1} |I|} = C2^{N/2} = C$. 

This gives us:
\begin{equation}
P(\mbox{rule 4 fails})\leq C(1+CB) e^{-CB^2/2}
\end{equation}
\end{proof}

We remark that rules 1,2 and 4 alone (not considering rule 3) give rise to an independent tree. The probability that a node fails is given by $p_f = \frac{C}{A}+C(1+CB) e^{-CB^2/2}$

\begin{lemma}
$$
P(\mbox{there is d-ary subtree, which survives rules }1,2,4)> 1-2p_{f}
$$
\end{lemma}
\begin{proof}
The process under scrutiny is dominated by a Galton Watson process because the death rules for one child is independent of the death rules for another child and the death probabilities are uniformly bounded by the estimates we have. The maximum number of descendats is $M = 2^{(N-4)/5}$.

Let $\tau$ be the probability that the GW tree $T$ has a d-ary subtree. Let $T_m$ be the tree $T$ truncated at level $m$ (i.e. after $m$ steps). Let $q_m$ be the probability that $T_m$ doesn't contain a d-ary subtree. Then 
\begin{eqnarray*}
\tau = \lim (1-q_m)
\end{eqnarray*}

In addition we have the recurrence relation: 
\begin{eqnarray*}
q_m=G_d(q_{m-1})
\end{eqnarray*}
where $G_d(s)$ is the probability that at most d-1 children are marked, when we mark each child independently (of each other and of the initial GW birth rules) with probability $1-s$. 

We have the following bound on $G_d(s)$:
\begin{equation}
G_d(s)\leq p_f+\sum_{j=0}^{d-1} C_M^j s^{M-j}
\end{equation}
If the probability, $p_f$, that an individual dies is small to start with (on the order of $C_M^{M/2} (2p_f)^{d-2} < 1$), then $ 1-\tau < 2p_f$.
\end{proof}

Another essential estimate is the following:
\begin{lemma}
$P(\mbox{there is a d-ary subtree, which survives rule 3}) >1-p_3$
\end{lemma}

\begin{proof}
We divide rule 3 in infinitely many rules. Rule $n$ is: $I$ survives if 
\begin{eqnarray}\label{SurvivalCondition3}
\sup_{\theta\in \mathcal{J}(I)}\sum_{J\subset \mathcal{J}(Ans_n((I))\setminus\mathcal{J}(I)}a_J\psi_J(\theta)  - a_J\psi_J(\theta_I)  \leq B 2^{-n}\\
\sup_{\theta\in \mathcal{J}(I)}\sum_{J\subset \mathcal{J}(Ans_n((I))\setminus\mathcal{J}(I)}a_J\psi_J(\theta)  - a_J\psi_J(\theta_I)  \geq -B 2^{-n}
\end{eqnarray}
where $\theta_I$ is the center of $I$. Notice that for the intervals of generation $0, 1,\ldots, n-1$ the rules are satisfied by default. In addition, the intervals on level $n$ are pretty much perfectly correlated with one another. The intervals on level $n+1$ are correlated in groups of $2^{(N-4)}/5$, but these groups are independent of one another. 

Let $T_m$ be the tree after $m$ steps. Let $q_{n,m}$ be the probability that $T_m$ has no d-ary subtree which survives rule $n$.  We seek a recurrence relation on $q_{n,m}$.

If $T_m$ has no d-ary subtree, one of two things must have gone wrong:
\begin{itemize}
\item the descendants up to level $n$ die (call this event $Bad_1$), or 
\item the descendants up to level $n$ survive, but less than $d$ of the generation 1 descendants have a d-ary subtree (call this event $Bad_2$).
\end{itemize}

There are a total of $M = 2^{N-4}/5$ possible descendants in generation 1 and $M^n$ in generation $n$. The probability one of these dies is
\begin{equation}
P(sup(\sum)> B 2^{-n})\leq (1+B 2^{n((N-4)/2-1)}) e^{-\frac{B^2 2^{2n((N-4)/2-1)}}{2}}=: p_n
\end{equation}
This follows from an argument similar to the one in lemma (\ref{DistribIneqt}). It's a consequence of the Borel-TIS inequality applied to a centered Gaussian field which covers dyadic levels up to $\sim 2^{-(n-1)(N-4)}$ over a set of size $\sim2^{-(N-4)}$ (and behaves thus as a centered Gaussian random variable with variance $C2^{-n(N-4)/2}$).

We can now conclude that $P(Bad_1)\leq p_n M^n$.

The second bad event can be bounded as follows:
\begin{equation}
P(Bad_2)\leq \sum_{i=0}^{d-1} C_n^i (q_{n,m-1})^{k-i}.
\end{equation}
This holds for two reasons. First, once we condition on the survival of generations $1,\ldots,n$, the subtrees starting at generation 1 nodes are independent. Secondly, although they don't have exactly the same law as the original tree, they are stochastically dominated by a tree which follows the same rules with random variables $a_J\sim N(0,t_c)$. 

We may thus write 
\begin{equation}
q_{n,m}\leq (1+B 2^{n((N-4)/2-1)}) e^{-\frac{B^2 2^{2n((N-4)/2-1)}}{2}}\left(\frac{2^{N-4}}{5}\right)^n + \sum_{i=0}^{d-1} C_n^i (q_{n,m-1})^{k-i}
\end{equation}

One can easily see that this recurrence relation implies that $q_n = \lim_m q_{n,m}\leq 2p_n M^n$, which in turn implies:
\begin{equation}
P(\mbox{no d-ary subtree, which survives rule 3})\leq 2\sum_n p_n M^n
\end{equation}
which can be made arbitrarily small by choosing $B$ large enough.
\end{proof}

Putting the last couple of lemmata together and using the notation $p=p_3 +p_{f}$ we get:
\begin{lemma}
$P(\mbox{there is a good d-ary surviving subtree}) >1-p.$
\end{lemma}
Notice that if we take $A\sim e^{5B}$ the probability $p<\frac{C}{A}$.
By "good" we mean that the distortion is controlled. Explicitely, the distortion is $D_{t_{k+1}} < C A A_{t_{k+1}} 2^N e^{5B}\sim CA^2A_{t_{k+1}}$ for levels $[2^{-n_k},2^{-n_{k-1}+1}]$. $C$ is a universal constant (depends only on the properties of the vaguelets).

%
%
%

\section{Modulus estimate}
\label{SectionModulusEstimate}
In this section we are concerned with a deterministic modulus estimate. This modulus estimate is similar to some of the estimates of J. C. Yoccoz on the local connectivity of the Mandelbrot set (see e.g. \cite{H92}, \cite{M00}). 

Assume we have a mapping $G:\BBC\rightarrow\BBC$ with distortion $\mu$ and an annulus $A$ centered on the real axis(for simplicity we will use a dyadic annulus of radii 1 and 2). Assume $\mu = 0$ in the lower half plane. 

Assume the distortion in the top part of the annulus is bounded by $D$ and that in the two sides we have connected sub-domains $\Omega_{1}$ and $\Omega_{2}$ of $A$. $\Omega_j$ is constructed in a stopping time fashion on a dyadic grid. At level $2^{-Ni}$ the number of surviving intervals is $d^i$ and the distortion there is bounded by $D_i$. One should think of $\Omega_j$  converging to a Cantor set $E_j$ on $\BBR\cap A$. See figure \ref{Modulus}.

\begin{figure}[h!]
  \centering
    \includegraphics{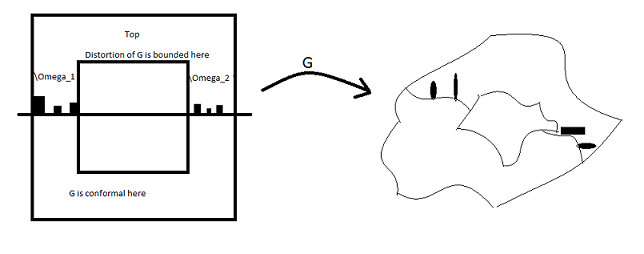}
 \caption[Distortion in annulus]{The annulus $A$ and the mapping $G$. The black regions are the ones where we have no control on the distortion.}
  \label{Modulus}
\end{figure}

The next theorem is a generalization of the following statement: if $G$ is conformal in the top of $A$ and inside $\Omega_1\cup\Omega_2$ then $mod(F(A))>c_0$ which depends on the logarithmic capacity of the sets $E_1, E_2$.

\begin{theorem}
$Mod(G(A))> \alpha = const.$
\end{theorem}
\begin{proof} One way to prove this result is to construct two closed curves inside $G(A)$ which are neither too long, nor too close to one another. When $G$ is conformal (as above) this is a consequence of Pfluger's theorem (see for example \cite{BB99} for a reference) and a modulus estimate. 

We will use a more direct argument here. We know that $mod G(A) = \frac{1}{mod G(\tilde \Gamma)}$, where $\tilde \Gamma$ is the family of curves which connect the two components of the complement of $G(A)$. We want to obtain an upper bound on $mod G(\tilde \Gamma)$ and we do this by constructing a good metric in G(A).

We consider the metric $\rho$ on $G(A)$: $\rho(w) = \frac{|\nabla u(z)|}{|G_{z}|-|G_{\overline{z}}|}$ where $G(z)=w$ and $u$ is defined below. 

Then $mod G(\tilde \Gamma) \leq \int\int_\Omega \rho^2 dA(w) \leq \int\int_{G(\Omega)} |\nabla u|^2 D(z) dA(z)$, where $D(z)$ is the distortion of $G$. 

For this argument to work we need to take $u$ such that $\int_{\gamma}|\nabla u| |dz| \geq 1$ for all $\gamma\in\tilde\Gamma$.
In $\Omega_2$ we take $u(x,y) = (\int_1^x f_y(t)dt,y)$ and similarly for $\Omega_1$. For each $y$ the density $f_y(t)$ puts all the mass uniformly on the surviving intervals defining $\Omega_j$ and zero in the regions where we have no control on the distortion. We may say that for $y\in [2^{-Ni}, 2^{-Ni+N}]$, $f_y(t) = 2^{Ni}d^{-i}$ on the intervals defining $\Omega_2$. Set $\nabla u = 1$ at all points where the distortion is bounded (the lower half of $A$  and the top of $A$). 

Now we have $\int_{\gamma} |\nabla u||dz|\geq |\int_{\gamma} \nabla u \cdot dz| \geq1$ so $|\nabla u|$ is an admissible metric.

Now we have that 
\begin{eqnarray*}
mod(G(\tilde\Gamma))\leq C \int_0^1\int_1^2 f_y^2(x) D(x,y) dx dy + C\\
\leq C\sum_i \int_{2^{-Ni}}^{2^{-N(i-1)}}\int_1^2 f_y^2(x)D_i(x) dx dy  + C\\
\leq C+ C\sum_i 2^{-Ni} D_i 2^{2Ni}d^{-2i} 2^{-Ni}d^{i} =C+ C\sum_i D_i d^{-i}
\end{eqnarray*}

So if $D_i$ increases slow enough the modulus will be bounded by a constant. 

\end{proof}

We will apply this theorem to an annulus where the distortion is controlled by $D_{t_{k+1}} < C A A_{t_{k+1}} 2^N e^{5B}\sim CA^2A_{t_{k+1}}$ for levels $[2^{-n_k},2^{-n_{k-1}+1}]$ to get $Mod(G(A(z,r,2r)))> \frac{C}{A^2 A_{k_0}}$. The constant $A_{k_0}$ corresponds to the level $r \sim 2^{-n_{k_0}}$. We are allowed to apply this theorem because the following relation holds
\begin{eqnarray}
\sum_{i=1}^{\infty} A_{k_0+i} d^{-\frac{n_{k_0+i}-n_{k_0+i-1}}{N}} < \infty \label{ModulusCondition}
\end{eqnarray}
which is a lot stronger than the necessary relation
\begin{eqnarray}
\sum_{i=1}^{\infty} A_{k_0+i} d^{-\frac{n_{k_0+i}-n_{k_0}}{N}} < \infty. 
\end{eqnarray}
Inequality (\ref{ModulusCondition}) holds because $A_{k}\sim C_{neg}^{k^{3\gamma k^\gamma} }$ while $n_{k}-n_{k - 1}\sim  C_{neg}^{k^{3\gamma k^\gamma} }$. In addition $d \sim 2^{N-5}$.

We conclude this section with the following summary: if we run the stopping time algorithm for an annulus at scale $2^{-n_{k_0}}$ and obtain the surviving d-ary tree, the distortion of the modulus of image of the annulus is given by $A^2A_{k_0}\sim A^3 C_{neg}^{k^{3\gamma k^\gamma}}$.

\section{Putting it all together}

\label{SectionMainProbabilisticEstimateRevisited}

We are now in a position to give an outline of the proof of the main probabilistic estimate. Fix $n$ and $z$. $\tilde\rho_n$ is fixed. We look of for $\rho_n, N_n,b_n$ and $c_n$ such that (\ref{MainProbEstimate}) holds.

If all of the following events hold for a large number of annuli
\begin{itemize}
\item $t_{i,k}, tr_{i,k}$ are small
\item There is an infinite d-ary surviving subtree in region $R_{i,1}$.
\item $L_{i,k},s_{i,k,j}$ are small,
\end{itemize}
the sum of the moduli is big. Small sum of moduli implies a large number of failures for at least one of these rules. We have probabilistic estimates for each such failure. We will treat each rule separately and due to the decoupling we will be able to use the independence of variables corresponding to some of the scales. 

We now proceed with the argument. 

Define the following indicator functions of events:
\begin{eqnarray*}
\chi_1(i) := \prod_{k=0}^{i-1}\chi(t_{i,k} < 2^{k-i}\ln D_n)\\
\chi_2(i) := \chi(tr_{i,i} < \ln D_n)\\
\chi_3(i) := \prod_{k>i}^{N_n}\chi(tr_{i,k}<2^{i-k}\ln D_n)\\
\chi_4(i) := \prod_{j=i+1}^{\infty}\prod_{k=i+1}^{\min\{N_n,j-1\}}\chi(s_{i,k,j}<2^{k-j}\ln (D_n 2^{k-i} ))\\
\chi_5(i) := \chi(L_{i,i} < D_n)\\
\chi_6(i) := \prod_{k=i+1}^{N_n}\chi(L_{i,k}<4^{i-k} D_n)\\
\chi_7(i) := \chi(\mbox{there is a good surviving d-ary subtree with }A^2A_{k_0}  = D_n)
\end{eqnarray*}
If $\prod_{l=1}^7\chi_l(i)  =1 $ then $Mod(G(A(z,\tilde\rho_n\rho_n^{i},2\tilde\rho_n\rho_n^{i}/2))) > cD_n^{-3}$, where $c$ is the constant from the modulus estimate and is independent of $n$ and $i$.

We divide all these events in three categories. Category I contains events of type 2, 5 and 7. Category II events are those of type 6 and all others are Category III. 

The event $\sum_{i=1}^{N_n}Mod(G(A_i)) < \alpha_n cD_n^{-3} N_n$ can happen for two reasons.
\begin{itemize}
\item[R1.] either there is a set of $2\alpha_n N_n$ of indices $i$ for which all events of category II and III take place, but more than $\alpha_n N_n$ events of category I don't;
\item[R2.] or no $2\alpha_n N_n$ indices $i$ exist for which all category II and III events take place. This can happen for two reasons. 
\begin{itemize}
\item[a.] either for more than $1-4\alpha_n N_n$ indices $i$ category II events do not take place, 
\item[b.] or for more than $2\alpha_n N_n$ indices $i$ category III events do not take place. 
\end{itemize}
\end{itemize}

The following inequalities hold because of the independence of the $\chi_l(i)$ for those particular values:
\begin{eqnarray}
P(R1) \leq C2^{N_n}\left((1+ \ln (D_n))e^{-\ln^2D_n/2}\right)^{\alpha_n N_n/3}\\
+ 2^{N_n}\left(\frac{C_n}{D_n}\right)^{\alpha_n N_n/3}+ 2^{N_n}\left(\frac{C}{\sqrt{D_n}}\right)^{\alpha_n N_n/3}
\end{eqnarray}
The variable $C_n$ is the same which appears in lemma (\ref{DistribIneqL}) and comes from the bound on the negative moments.

We get a bound on $P(R2.a) $ by following the argument in \cite{AJKS09} using lemmata \ref{DistribIneqL} and \ref{DistribIneqt}:
 $$
2^{N_n}2C_nD_n^{-1} \rho_n^{-(2+3a_n/2)}\rho_n^{(1+2b_n)N_n}
$$
 whenever $\rho_n$ is chosen small enough to have $4\rho_n^{a_n/2} < 1/2$ and $D_n$ is large enough to have $C_nD_n^{-1}\rho_n^{-(2+3a_n/2)} <1/2$ ($C_n$ is the constant which appears in lemma (\ref{DistribIneqL})). In addition $\alpha_n$ and $b_n$ have to satisfy $1+2b_n = (1+a_n/2)(1-4\alpha_n)$.

We present here the estimates for the sake of completion:
\begin{eqnarray}
P(R2.a)\leq \sum_{|S|=p}E[\prod_{i\in S} \chi_6^C(i)]
\end{eqnarray}
where $p=1-\alpha_nN_n$ and $S$ is the set of indices where the events of type 6 fail. We introduce the notation
$$
\chi_6(i,k) := \chi(L_{i,k}<4^{i-k} D_n)
$$
The failure of a $\chi_6(i)$ is due to some $k$, but it is possible that two different $i$'s fail because of the same $k$. We have the bound:
\begin{eqnarray}\label{EstimateType6Decoupled}
E[\prod_{i\in S} \chi_6^C(i)] \leq \sum_{r=1}^p\sum_{(l_1,\ldots ,l_r)}E[\prod_{j=1}^r \chi_6^C(i_j, i_j+l_j)]
\end{eqnarray}
where $i_1$ is the smallest $i\in S$ such that $\chi_6(i_1) = 0$ and level $i_1 + l_1$ causes $\chi_6(i_1)$ to fail. More generally, $i_{j+1}$ is the smallest index $i\in S$ larger than $i_j+l_j$ and level $i_{j+1}+l_{j+1}$ causes $\chi_6 (i_{j+1})$ to fail. 

The intervals $[i_j, i_j +l_j]$ cover the set $S$ and this implies 
$$
\sum_{j=1}^r l_j \geq p-r
$$
Applying lemma \ref{DistribIneqL} we get:
\begin{eqnarray}
E[\chi_6^C(i_j, i_j+l_j)]\leq C_n 4^{l_j}D_n^{-1}\rho_n^{(l_j-1)(1+a_n)}= C_nD_n^{-1} \rho_n^{-(1+a_n)} \left(4\rho_n^{1+a_n}\right)^{l_j}
\end{eqnarray}

The terms on the right hand side of (\ref{EstimateType6Decoupled}) are independent so we get:
\begin{eqnarray}
E[\prod_{i\in S} \chi_6^C(i)] \leq  \sum_{r=1}^p \left(C_n D_n^{-1}\rho_n^{-(1+a_n)}\right)^r\sum_{(l_1,\ldots ,l_r)}  \left(4\rho_n^{1+a_n}\right)^{\sum l_j} \\
\leq \rho_n^{(1+a_n/2)p}\sum_{r=1}^p \left(C_nD_n^{-1} \rho_n^{-(2+3a_n/2)}\right)^r\sum_{(l_1,\ldots ,l_r)}  \left(4\rho_n^{a_n/2}\right)^{\sum l_j}
\end{eqnarray}
where we have used the fact that $\sum l_j \geq p-r$. We now drop the constrains on $l_j$ to get:
\begin{eqnarray}
E[\prod_{i\in S} \chi_6^C(i)] \leq \rho_n^{(1+a_n/2)p} \sum_{r=1}^p \left(C_nD_n^{-1} \rho_n^{-(2+3a_n/2)}\right)^r\left(\sum_{l=1}^\infty  \left(4\rho_n^{a_n/2}\right)^{l}\right)^r
\end{eqnarray}
If we take $\rho_n$ small enough to have $4\rho_n^{a_n/2}<1/2$ and $D_n$ large enough to have $C_nD_n^{-1} \rho_n^{-(2+3a_n/2)} <1/2$ we get:
\begin{eqnarray}
E[\prod_{i\in S} \chi_6^C(i)] \leq 2C_nD_n^{-1} \rho_n^{-(2+3a_n/2)}\rho_n^{(1+a_n/2)p}\leq  \rho_n^{(1+a_n/2)p}
\end{eqnarray}

If we now take $\alpha_n$ small enough such that $1+2b_n = (1+a_n/2)(1-4\alpha_n)$ we get 
\begin{eqnarray}
E[\prod_{i\in S} \chi_6^C(i)] \leq 2C_nD_n^{-1} \rho_n^{-(2+3a_n/2)}\rho_n^{(1+2b_n)N_n}
\end{eqnarray}
and 
\begin{eqnarray}
P(R2.a)\leq 2^{N_n}2C_nD_n^{-1} \rho_n^{-(2+3a_n/2)}\rho_n^{(1+2b_n)N_n}
\end{eqnarray}

This concludes our bound on the failure of events of type 6.

Failure of events of type 1 and 3 can be bounded as in  \cite{AJKS09}. We use a similar argument now to take care of events of type 4. The estimate we get for type 4 events also bounds the failure of type 1 and 3 events.

The probability that more than $\alpha_nN_n$ events of type 4 fail is bounded by
\begin{eqnarray}
\sum_{|S|=p}E[\prod_{i\in S} \chi_4^C(i)]
\end{eqnarray}
where $p = \alpha_n N_n$ and $S$ is a subset of indices $i$ of cardinality $p$. The letter $C$ in the exponent of the indicator function refers to the complement/failure of the respective event. 

Each variable $\chi_4(i) = 0$ because of some bad event on level $j$. Although $i's$ are different, the problematic levels $j$ might not be; the same problematic level might cause problems for several $i$. This leads us to the following bound:
\begin{eqnarray}
E[\prod_{i\in S} \chi_4^C(i)] \leq \sum_{r=1}^p \sum_{m_1+\ldots+m_r = p}\left(\sum_{j_1 = i_1+m_1}^{\infty}\chi_4^C(i_1,j_1)\right)\cdots\left(\sum_{j_r = i_r+m_r}^{\infty}\chi_4^C(i_r,j_r)\right)\nonumber\\
\mbox{where  }\chi_4(i,j) = \prod_{k=i}^{\min\{N_n,j-1\}}\chi(s_{i,k,j}<2^{k-j}\ln (D_n 2^{k-i} ))\nonumber
\end{eqnarray}
The bound should be read as follows: there can be between $1$ and $p$ problematic levels $j$. Say there are $r$ problematic levels $j$. The total set of indices $S$ will be partitioned in $r$ subsets of cardinality $m_1,\ldots, m_r$. If a level $j$ causes problems for $m$ indices $i$ there must be an index $i$ such that $i+m = j$.

We now have the following:
\begin{eqnarray*}
E[\sum_{j_l= i_l+m_l}^{\infty}\chi_4^C(i_l,j_l)]\leq \sum_{j_l= i_l+m_l}^{\infty}\sum_{k=i_l}^{\min\{N_n,j_l-1\}} e^{-\ln^2(D_n 2^{k-i_l})2^{2(k-j_l)}\rho_n^{2(k-j_l)q/2}/4}
\end{eqnarray*}
If $i_l+m_l \leq N_n$, this is bounded by
\begin{eqnarray*}
 \sum_{j_l= i_l+m_l}^{N_n}\sum_{k=i_l}^{j_l-1} +   \sum_{j_l= N_n+1}^{\infty}\sum_{k=i_l}^{N_n}
\end{eqnarray*}
Otherwise (if $i_l+m_l \geq N_n+1$) it is bounded by 
\begin{eqnarray*}
  \sum_{j_l= i_l+m_l}^{\infty}\sum_{k=i_l}^{N_n} e^{-\ln^2(D_n 2^{k-i_l})2^{2(k-j_l)}\rho_n^{2(k-j_l)q/2}/4} \leq Ce^{-\ln^2(D_n2^{m_l})\rho_n^{-q}}
\end{eqnarray*}
Similarly, the second sum above is dominated by the term in which $j_l=N_n+1$ and $k=N_n$ 
\begin{eqnarray}
 \sum_{j_l= N_n+1}^{\infty}\sum_{k=i_l}^{N_n} \leq Ce^{-\ln^2(D_n2^{N_n-i_l})\rho_n^{-q}} \leq Ce^{-\ln^2(D_n2^{m_l})\rho_n^{-q}}
\end{eqnarray}
The first sum is dominated by
\begin{eqnarray}
\sum_{j_l= i_l+m_l}^{N_n}\sum_{k=i_l}^{j_l-1}  \leq \sum_{j_l= i_l+m_l}^{N_n} e^{-\ln^2(D_n2^{j_l-i_l})\rho_n^{-q}}\\
\leq Ce^{-\ln^2(D_n2^{m_l})\rho_n^{-q}} \leq  Ce^{-\ln^2(D_n)\rho_n^{-q}  - 2\ln 2\ln(D_n) m_l\rho_n^{-q} -4 m_l^2\rho_n^{-q}}
\end{eqnarray}
Then we have 
\begin{eqnarray}
 \sum_{r=1}^p\sum_{m_1+\ldots+m_r = p} \left(\ldots\right)\ldots\left(\ldots\right)\leq \\
 \sum_{r=1}^p e^{-r\ln^2(D_n)\rho_n^{-q}}e^{-2p\ln 2\ln(D_n)  \rho_n^{-q}} \left(\sum_{m=1}^{p}Ce^{-4 m^2\rho_n^{-q}}\right)^r\\
\end{eqnarray}
Here $C$ is a universal constant and $\rho_n$ is small which implies the rightmost term is less than $2^r$. We thus get:
\begin{eqnarray}
E[\prod_{i\in S} \chi_4^C(i)] \leq C e^{-2p\ln 2\ln(D_n)  \rho_n^{-q} - \ln^2(D_n)\rho_n^{-q} }
\end{eqnarray}
So the probability that more than $\alpha_nN_n$ events of type 4 fail is less than:
\begin{equation}
C2^{N_n} e^{-2\alpha_nN_n\ln 2\ln(D_n)  \rho_n^{-q} - \ln^2(D_n)\rho_n^{-q} }
\end{equation}

We conclude this section with the following inequality:
\begin{eqnarray}
P(\sum_{i=1}^{N_n}Mod(G(A_i)) < \alpha_n cD_n^{-3} N_n)\leq   C2^{N_n}\left((1+ \ln (D_n))e^{-\ln^2D_n/2}\right)^{\alpha_n N_n/3}\\
+ 2^{N_n}\left(\frac{C_n}{D_n}\right)^{\alpha_n N_n/3} + 2^{N_n}\left(\frac{C}{\sqrt{D_n}}\right)^{\alpha_n N_n/3}\\
+  2^{N_n}C_nD_n^{-1} \rho_n^{-(2+3a_n/2)}\rho_n^{(1+2b_n)N_n}\\
+ C2^{N_n} e^{-2\alpha_nN_n\ln 2\ln(D_n)  \rho_n^{-q} - \ln^2(D_n)\rho_n^{-q} }
\end{eqnarray}
as long as all the constants satisfy the necessary inequality given above.

\subsection{Final estimates}

We now proceed with an analysis of the order of magnitude of all the constants involved such that all the results hold and such that  
\begin{equation}
P(\sum_{i=1}^{N_n}Mod(G(A_i)) < \alpha_n cD_n^{-3} N_n)\leq \tilde\rho_n\rho_n^{(1+b_n)N_n}
\end{equation}

\begin{itemize}
\item $a_n$ is given to us by the size of the variance in the GFF. It's size is $a_n\sim n^{-3\gamma}$. This forces us to take $\alpha_n \sim a_n/2$ and $b_n\sim n^{-3\gamma}$ (earlier we had $k$ in stead of $n$).
\item We also take $C_n \leq D_n^{1/2}$. Since $C_n \sim C_{neg}^{n^{3\gamma n^\gamma}}$ we must have $D_n \geq C_{neg}^{2n^{3\gamma n^\gamma}}$.
\item The inequality $4\rho_n^{a_n/2} < 1/2$. Since $b_n$ is less than $a_n$ it suffices to take  $4\rho_n^{a_n/2} < 1/2$ which leads us to $\rho_n \sim 2^{-10 n^{3\gamma}}$. 
\item We also need to have $C_nD_n^{-1} \rho_n^{-(2+3a_n/2)}<1/2 $. If we take $D_n \geq C_{neg}^{2n^{3\gamma n^\gamma}}$, this is automatically satisfied.
\item It suffices to have the following inequalities (these two terms dominate the others):
\begin{eqnarray}
 2^{N_n}\left(\frac{C}{\sqrt{D_n}}\right)^{\alpha_n N_n/3} \leq \tilde\rho_n\rho_n^{(1+b_n)N_n}\\
2^{N_n}C_nD_n^{-1} \rho_n^{-(2+3a_n/2)}\rho_n^{(1+2b_n)N_n}  \leq \tilde\rho_n\rho_n^{(1+b_n)N_n}
\end{eqnarray}
The second inequality can be rewritten as:
\begin{eqnarray}
2^{N_n}C_nD_n^{-1} \rho_n^{-(2+3a_n/2)} \rho_n^{b_nN_n}  \leq \tilde\rho_n
\end{eqnarray}
\item For consistency we also need $\tilde\rho_n\sim 2^{-n_n}\sim \rho_1^{N_1}\ldots\rho_{n-1}^{N_{n-1}}\sim 2^{-10(N_1+\ldots + (n-1)^{3\gamma}N_{n-1})}$.
\item In order to obtain the desired modulus of continuity of $f_+$ we need to have that $\alpha_n N_n D_n^{-3}$ is greater than a positive power of $N_n$.
\item If we set  $D_n \sim C_{neg}^{2n^{3\gamma n^\gamma}}$ the relations above become 
\begin{eqnarray}
2^{N_n}C_{neg}^{-n^{3\gamma n^\gamma}\alpha_n N_n/3}\leq 2^{-10(N_1+\ldots + (n-1)^{3\gamma}N_{n-1})}2^{-10n^{3\gamma}(1+b_n)N_n}\\
C_{neg}^{-n^{3\gamma n^\gamma}}2^{-N_n}\leq  2^{-10(N_1+\ldots + (n-1)^{3\gamma}N_{n-1})}
\end{eqnarray}
given the way we chose $\rho_n$.
\item Since $N_n>\alpha_n^{-1} D_n^3\sim n^{3\gamma}C_{neg}^{6n^{3\gamma n^\gamma}}$ it suffices to take:
\begin{eqnarray}
N_n \geq 10N_1+\ldots + 10(n-1)^{3\gamma}N_{n-1}
\end{eqnarray}
This condition is satisfied by any sequence $N_n$ which is a power (larger than 2) of $C_{neg}^{6n^{3\gamma n^\gamma}}$.
\item Take $N_n$ such that $\alpha_nN_n D_n^{-3} \sim N_n^{1-\epsilon}$ or $N_n\sim C_{neg}^{18 n^{3\gamma n^\gamma}/\epsilon}$ for $1/4>\epsilon > 0$ (independent of $n$). Then
\begin{eqnarray}
\sum_{i=1}^{N_n}Mod(G(A_i))\geq N_n^{1-\epsilon}.
\end{eqnarray}
This implies the modulus of continuity is given by the relationship of $\omega(t)\sim e^{-\sum_{i=1}^{n-1} N_i^{1-\epsilon}}$ and $t\sim2^{- \sum_{i=1}^n i^{3\gamma}N_i}$. 
\item For $t$ in the range $[2^{- \sum_{i=1}^n i^{3\gamma}N_i}, 2^{- \sum_{i=1}^{n-1} i^{3\gamma}N_i}]$ we get $\omega(t) \leq \omega (2^{- \sum_{i=1}^{n-1} i^{3\gamma}N_i})\leq e^{-\sum_{i=1}^{n-1} N_i^{1-\epsilon}}$.
\item We want this to be less than $e^{-(\log\frac{1}{t})^{1-2\epsilon}}$. This implies we need to have the following:
\begin{eqnarray}
e^{-\sum_{i=1}^{n-1} N_i^{1-\epsilon}} \leq e^{-(\log\frac{1}{2^{- \sum_{i=1}^n i^{3\gamma}N_i}})^{1-2\epsilon}}
\end{eqnarray}
which is equivalent to:
\begin{eqnarray}
(\log\frac{1}{2^{- \sum_{i=1}^n i^{3\gamma}N_i}})^{1-2\epsilon} \leq \sum_{i=1}^{n-1} N_i^{1-\epsilon} \mbox{ or }\\
(\sum_{i=1}^n i^{3\gamma}N_i)^{1-2\epsilon} \leq \sum_{i=1}^{n-1} N_i^{1-\epsilon} 
\end{eqnarray}
\item The last terms on each side dominate the computation, so we need to have $(n^{3\gamma}N_n)^{1-2\epsilon} < N_{n-1}^{1-\epsilon}$.
\item If we set $N_n = e^{f(n)}$ the last inequality becomes $(1-2\epsilon)(f(n) +3\gamma\ln n) <(1-\epsilon)f(n-1)$. Rewriting this we get $f(n) +3\gamma\ln n < 1+ \frac{\epsilon}{1-2\epsilon}f(n-1)$. 
\item If $\frac{\ln x}{f(x)}\rightarrow 0$ as $n\rightarrow \infty$ (this is the case in our situation), we get the inequality above for any function $f(x)$ for which $\frac{f'(x)}{f(x)}< \epsilon $ for $x$ large.
\item In our situation $f(x) \sim x^{3\gamma x^\gamma}/\epsilon = e^{3\gamma x^{\gamma}\ln x /\epsilon}$ and hence $\frac{f'(x)}{f(x)} = \frac{3\gamma}{\epsilon} (\gamma x^{\gamma-1}\ln x + x^{\gamma - 1})$.
\item If $\gamma < 1$ for all $x$ large enough $\frac{f'(x)}{f(x)}< \epsilon $, which implies that we get the modulus of continuity $e^{-(\log\frac{1}{t})^{1-2\epsilon}}$
\end{itemize}

Recall that we needed the following estimate in the proof of theorem \ref{SolutionRandomWeldingProblem} in section \ref{SectionConformalWelding}.

\begin{lemma}\label{IntegrabilityOfDistortion}
Let $D(z)$ be the distortion inside $\BBD$. Then
\begin{equation} 
E[\int_\BBD DdA]<\infty
\end{equation}
\end{lemma}
\begin{proof}
In each Whitney square the distortion $D$ is bounded by the constant $C_n$ (from the proof of lemma \ref{DistribIneqL}) for scales $[2^{-n_{n+1}}, 2^{-n_n}]$. Then we have 
\begin{eqnarray}
E[\int_\BBD DdA] \leq \int_\BBD E[D]dA \leq C\sum_n 2^{-n_n} C_n\leq C \sum_n 2^{-n_n} C_{neg}^{n^{3\gamma n^\gamma}}<\infty
\end{eqnarray}
\end{proof}

We conclude by the following summary: for the sequence $t_n = 2- \frac{1}{n^\gamma}, \gamma < 1$, and for any $\epsilon > 0$, we consider the sequence $n_n  = \sum_{i=1}^{n} i^{3\gamma}N_i $, where $N_i = C_{neg}^{18 i^{3\gamma i^\gamma}/\epsilon}$. If we construct the random measure using this sequence we get a conformal welding and a modulus of continuity of $e^{-(\log\frac{1}{t})^{1-2\epsilon}}$. 

For other sequences $\{t_n\}$  which converge to the critical value, we still get a conformal welding, but we can not prove the uniqueness by means of the removability theorems of Jones/Smirnov and Koskela/Nieminen. 

Finally, for $t_n\rightarrow 2$, the best modulus of continuity for the welding map is going to be worse than Holder.



\begin{thebibliography}{9999999}
%
\bibitem[A06]{A06} Ahlfors, L. \textit{Lectures on quasiconformal mappings}, 2nd edition, AMS, 2006.

\bibitem[AJKS09]{AJKS09} Astala K., Jones P.W., Kupiainen A., Saksman E. \textit{Random conformal weldings} preprint, 2009.

\bibitem[AIM09]{AIM09} Astala, K., Iwaniec T., Martin G., \textit{Elliptic partial differential equatons and quasiconformal mappings in the plane}, Princeton University Press, 2009.

\bibitem[BNB00]{BNB00} Bachman, G., Narici, L., Beckenstein E. \textit{Fourier and wavelet analysis} Springer-Verlag New York, 2000.

\bibitem[BM03]{BM03} Bacry, E., Muzy, J.F. \textit{Log-infinitely divisible multifractal processes} Comm. Math. Physics. 236, 2003, 449-375.

\bibitem[BB09]{BB99} Balogh, Z., Bonk, M. \textit{Lengths of radii under conformal maps of the unit disc} Proc. Amer. Math. Soc. {\bf 127}, 3, 1999, 801-804.

\bibitem[BS12]{BS12} Binder, I., Smirnov, S. \textit{Personal communication by I. Binder}, 2012.

\bibitem[B94]{B94} Bishop, C., \textit{Some homeomorphism of the sphere conformal off a curve} Ann. Acad. Sci. Fenn. Ser. A I Math. {\bf 19}, 1994, 323-338.

\bibitem[B06]{B06} Bishop, C., \textit{Conformal welding and Koebe's theorem} preprint, 2006. To appear in Annals of Math.

\bibitem[B66]{B66} Burkholder, D. L. \textit{Martingale transforms} Ann of Math Stat {\bf 37}, 6, 1966, 1494-1504.

\bibitem[D88]{D88} David, G. \textit{Solutions de l'equation de Beltrami avec $||\mu||_\infty=1$} Ann. Acad. Sci. Fenn. Ser A I Math. \{bf 13\},  1988, 25-70.

\bibitem[D95]{D95} Donoho, David, L. \textit{Nonlinear solution of linear inverse problems by wavelet-vaguelette decomposition} Applied and Computational Harmonic Analysis {\bf 2}, 1995, 101-126.

\bibitem[DS11]{DS11}Duplantier, B., Sheffield, S. \textit{Liouville quantum gravity and KPZ} Invent. math. 185, 2011, 333–393.


\bibitem[H91]{H91} Hamilton, D.H. \textit{Generalized conformal welding} Ann. Acad. Sci. Fenn. Ser. A I Math \textbf{16}, 1991, 333-343.



\bibitem[H92]{H92} Hubbard, J. H. \textit{Local connectivity of Julia sets and bifurcation loci: three theorems of J.-C. Yoccoz} Topological Methods in Modern Mathematics (Stony Brook, NY, 1991) 467-511, Publish or Perish, Houston, TX, 1993.


\bibitem[J12]{J12} Jones, Peter W. \textit{Private communication}, 2012.

\bibitem[JS00]{JS00} Jones, Peter W. and Smirnov, Stanislav S., \textit{ Removability theorems for Sobolev functions and quasiconformal maps} Ark. Mat. {\bf 38}, 2000, 263-279.

\bibitem[JM95]{JM95} Jones, Peter W. and Makarov, Nikolai G., \textit{Density properties of harmonic measure} Ann. of Mathematics, Second Series, {\bf 142}, 1995, 427-455.




\bibitem[K85]{K85} Kahane, J.-P., \textit{Sur le chaos multiplicatif} Ann. Sci. Math. Quebec 9, 1985,  435-444.

\bibitem[KP76]{KP76} Kahane, J.-P., Peyriere, J. \textit{Sur certaines martingales de Benoit Mandelbrot} Advances in Math. 22, 1976, 131-145.

\bibitem[KN05]{KN05} Koskela, Pekka; Nieminen, Tomi \textit{Quasiconformal removability and the quasihyperbolic metric} Indiana Univ. Math. J. {\bf 54}, no. 1,  2005, 143--151.

\bibitem[L05]{L05} Lawler, G. F. \textit{Conformally invariant processes in the plane} AMS, 2005.

\bibitem[LSW04]{LSW04} Lawler, G. F., Schramm, O., Werner W.\textit{Conformal invariance of planar loop-erased random walks and uniform spanning trees} Annals of Prob. 32, 939-995.

\bibitem[L70]{L70} Lehto, O. \textit{Homeomorphisms with a given dilatation} Lecture Notes in Mathematics, Vol. 118 Springer, Berlin,  58–73.  

\bibitem[LV73]{LV73} Lehto, O., Virtanen K.I. \textit{Quasiconformal mappings in the plane}, Springer, 1973.

\bibitem[M74]{M74} Mandelbrot, B. B. \textit{Intermittent turbulence in self-similar cascades:divergence of high moments and dimension of carrier} Journal of Fluid Mechanics 62, 1974, 331-358.


\bibitem[MFC97]{MFC97} Mandelbrot, B. B., Fisher, A., Calvet, L. \textit{The Multifractal Model of Asset Returns} Cowles Foundation discussion paper no. 1164, Yale University, paper available from the SSRN database at http://www.ssrn.com, 1997.


\bibitem[M90]{M90} Meyer, Y. \textit{Ondelettes et Operateurs} Herman Editeurs des sciences et des arts., 1990. 

\bibitem[MC97]{MC97} Meyer, Y., Coifman, R. R. \textit{Wavelets. Calderon-Zygmund and multilinear operators} Cambridge University Press, 1997.

\bibitem[M00]{M00} Milnor, J. \textit{Local connecticity of Julia sets: expository lectures} in \textit{The Mandelbrot set, Theme and Variations} Cambridge University Press, 2000.

\bibitem[O61]{O61} Oikawa, K. \textit{Welding of polygons and the type of Riemann surfaces} Kodai Math. Sem. Rep., 13, 1961, 37-52.

\bibitem[RV08]{RV08} Robert, R., Vargas, C. \textit{Gaussian multiplicative chaos revisited} ArXiv [math.PR] 0807.1030, 2008.

\bibitem[Sch00]{Sch00} Schramm, O. \textit{Scaling limits of loop-erased random walks and uniform spanning trees} Israel J. Math. 118, 2000, 221-288.

\bibitem[SS09]{SS09} Schramm, O., Sheffield, S. \textit{Contour lines of the two dimensional discrete Gaussian free field} Acta Math, 202(1), 2009, 21-137.

\bibitem[SS10]{SS10} Schramm, O., Sheffield, S. \textit{A contour line of the continuum Gaussian free field} Arxiv e-prints, 2010, 1008.2447.

\bibitem[S07]{S07} Sheffield, S. \textit{Gaussian free fields for mathematicians} Probab.  Theory Related Fields, 139(3-4), 2007, 521-541.

\bibitem[S10]{S10} Sheffield, S. \textit{Conformal weldings of random surfaces: SLE and the quantum gravity zipper}, arxiv:1012.4797v1, 2010.

\bibitem[Sm01]{Sm01} Smirnov, S. \textit{Critical percolation in the plane:Conformal invariance, Cardy's formula, scaling limits} C.R. Acad.Sci.Paris S. I Math. 333, no. 3, 239-244.


\bibitem[V85]{V85} Vainio, J. V. \textit{Conditions for the posibility of conformal sewing} Ann. Acad. Sci. Fenn. Ser. A I Math. Dissertationes 53, 1985, 43.

\bibitem[V89]{V89}  Vainio, J. V. \textit{On the type of sewing functions with a singularity} Ann. Acad. Sci. Fenn. Ser. A I Math. 14 (1), 1989,  161-167.

\bibitem[V95]{V95}  Vainio, J. V. \textit{Properties of real sewing functions} Ann. Acad. Sci. Fenn. Ser. A I Math. 20 (1), 1995,  87-95.



\bibitem[W09]{W09} Werner, W. \textit{Percolation et modele d'Ising}, Societe Mathmatique de France, 2009.




%







\end{thebibliography}
\end{document}